\documentclass[a4paper,reqno]{amsart}
\usepackage{graphicx}

\usepackage{hyperref}

\usepackage{amssymb,amsxtra,amsfonts,dsfont}

\usepackage{amsmath}

\usepackage[svgnames]{xcolor}

\usepackage[color,matrix,arrow]{xy}

\input xy \xyoption{all}

\newtheorem{lem}{Lemma}[section]

\newtheorem{thm}[lem]{Theorem}

\newtheorem{pro}[lem]{Proposition}

\newcommand{\slrw}[1]{\stackrel{#1}{\longrightarrow}}

\newcommand{\lrw}{\longrightarrow}

\newcommand{\LL}{\Lambda}

\newcommand{\xa}{\alpha}
\newcommand{\xb}{\beta}
\newcommand{\xc}{\gamma}
\newcommand{\xl}{\lambda}
\newcommand{\be}{\mathbf e}
\newcommand{\bi}{\mathbf i}
\newcommand{\bj}{\mathbf j}
\newcommand{\bv}{\mathbf v}
\newcommand{\bzero}{\mathbf 0}

\newcommand{\caD}{\mathcal D}
\newcommand{\caDb}{\mathcal D^b}

\newcommand{\caM}{\mathcal M}

\newcommand{\caK}{\mathcal K}

\newcommand{\RHom}{\mathbf{R}\strut\kern-.2em\operatorname{Hom}\nolimits}

\newcommand{\zZ}{{\mathbb Z}}

\newcommand{\tL}{\widetilde{\LL}}

\newcommand{\dtL}{\Delta{\LL}}

\newcommand{\dtnL}{\Delta_{\nu}{\LL}}

\newcommand{\ohL}{\widehat{\LL}}
\newcommand{\ohG}{\widehat{\GG}}

\newcommand{\diag}{\mathrm{diag}}

\newcommand{\GG}{\Gamma}
\newcommand{\tG}{\widetilde{\Gamma}}
\newcommand{\trho}{\tilde{\rho}}

\newcommand{\tP}{\tilde{P}}

\newcommand{\wht}{\widehat}

\newcommand{\tQ}{\widetilde{Q}}

\newcommand{\olQ}{\overline{Q}}

\newcommand{\om}[1]{\Omega^{#1}}

\newcommand{\sk}[1]{^{(#1)}}

\newcommand{\xrt}[1]{\xi_{r_{#1}}^{i_{#1}}}
\newcommand{\xrmt}[1]{\xi_{r_{#1}}^{-i_{#1}}}

\newcommand{\zzs}[1]{{\mathbb Z|_{#1} }}

\newcommand{\vv}{\diamond}

\newcommand{\mmod}{\mathrm{mod}\,}

\newcommand{\grm}{\mathrm{gr}\,}

\newcommand{\bam}{{\overline{m}}}

\newcommand{\dmn}{\mathrm{dim}\,}

\newcommand{\dmv}{\mathbf{dim}\,}

\newcommand{\gkdim}{\mathrm{GKdim}\,}

\newcommand{\ldmv}{\mathbf{l.dim}\,}

\newcommand{\arr}[2]{\begin{array}{#1}#2\end{array}}

\newcommand{\mat}[2]{\left(\begin{array}{#1}#2\end{array}\right)}

\newcommand{\eqqc}[2]{\begin{equation}\label{#1}#2\end{equation}}

\newcommand{\eqqcn}[2]{\[ #2 \]}
\newcommand{\eqqcnn}[2]{$ #2 $}

\begin{document}
\title{On $n$-slice Algebras and Related Algebras}
\author{Jin Yun Guo, Cong Xiao and Xiaojian Lu}
\address{Jin Yun Guo, Cong Xiao and Xiaojian Lu\\LCSM( Ministry of Education), School of Mathematics and Statistics\\ Hunan Normal University\\ Changsha, Hunan 410081, P. R. China}

\thanks{This work is supported by Natural Science Foundation of China \#11671126, and by the Construct Program of the Key Discipline in Hunan Province}

\begin{abstract}
The $n$-slice algebra is introduced as a generalization of path algebra in  higher dimensional representation theory.
In this paper, we give a classification of $n$-slice algebras via their $(n+1)$-preprojective algebras  and the trivial extensions of their quadratic duals.
One can always relate tame $n$-slice algebras to the McKay quiver of a finite subgroup of $\mathrm{GL}(n+1, \mathbb C)$.
In the case of $n=2$, we describe the relations for the $2$-slice algebras related to the McKay quiver of finite Abelian subgroups of $\mathrm{SL}(3, \mathbb C)$ and of the finite subgroups obtained from embedding $\mathrm{SL}(2, \mathbb C)$ into $\mathrm{SL}(3,\mathbb C)$.
%Insert your abstract here. Include keywords, PACS and mathematical
%subject classification numbers as needed.

%\keywords{$n$-slice algebra \and $(n+1)$-preprojective algebra \and McKay quiver \and $n$-translation algebra }
% \PACS{PACS code1 \and PACS code2 \and more}
% \subclass{ \and MSC code2 \and more}
%\subclass{16G20 \and 16G60 \and  16S35 \and 20C05}
\end{abstract}

\maketitle

\section{Introduction}
\label{intro}
Path algebras are very important in representation theory and related mathematical fields, it is interesting to study their counterparts in higher representation theory \cite{i11,hio14}.
The $n$-slice algebra is introduced and studied in \cite{g19b} as the quadratic dual of an $n$-properly-graded algebra whose trivial extension is an $n$-translation algebra (see Section \ref{sec:pre} for the details).
The $n$-slice algebra was  called dual $\tau$-slice algebras in \cite{g20,gx20}, for studying algebras related to the higher representation theory introduced and studied  by Iyama and his coauthors \cite{i11,io11,hio14}. The $n$-slice algebras bear some interesting features of path algebras in higher dimensional representation theory setting: the $\tau_n$-closures, higher dimensional counterparts of preprojective and preinjective components, are truncations of $\zZ|_{n-1} Q$, a higher version of $\zZ Q$ \cite{g19b}, and the $n$-APR tilts and colts can be realized by $\tau$-mutations  in  $\zzs{n-1} Q^{\perp,op}$ \cite{gx20}, which is obtained by removing a source (resp. sink) and adding a corresponding sink (resp. source).

\medskip

It is well known that the path algebras of an acyclic quiver are classified as finite, tame and wild representation types according to their quivers.
This classification has a great influence in representation theory.
When Iyama and his coauthors study $n$-hereditary algebras, they define and study $n$-representation-finite and $n$-representation-infinite algebras, they also define and study $n$-representation-tame algebras inside $n$-representation-infinite algebras \cite{io11,io13,hio14}.
Recently, Lee, Li, Mills, Schiffler and Seceleanu give a new characterization of such classification for acyclic quivers in terms of their frieze varieties \cite{llmss20}. The classification of path algebras can also be read out from their preprojective algebras, that is, a path algebra is of finite representation type if and only if its preprojective algebra is finite dimensional, is of tame representation type if and only if its preprojective algebra is of finite Gelfand-Kirilov dimension, and is of wild representation type if and only if its preprojective algebra is of infinite Gelfand-Kirilov dimension \cite{dr81}.

Regarding $n$-slice algebra as a higher version of path algebra, in this paper, we generalize the classification of path algebras via  $(n+1)$-preprojective algebras and related $n$-translation algebras.
By using the results in \cite{gw00}, we get that for an $n$-slice algebra, its $(n+1)$-preprojective algebra is either of finite dimension, of finite Gelfand-Kirilov dimension or of infinite Gelfand-Kirilov dimension.
The trivial extension of the  quadratic dual of an $n$-slice algebra is a stable $n$-translation algebra, so it is either periodic, of finite complexity and of infinite complexity, accordingly.
In this way we classify the $n$-slice algebras as finite type, tame type or wild type according to their $(n+1)$-preprojective algebras or according to the trivial extensions of their quadratic duals.

\medskip

McKay quiver for a finite subgroup of a general linear group is introduced in \cite{m79}, it connecting representation theory, singularity theory and many other mathematical fields.
McKay quiver are also very interesting in studying higher representation theory of algebras \cite{hio14,g19b}.
Over the algebraically closed field of characteristic zero, the preprojective algebras of path algebras of tame type are Morita equivalent to the skew group algebras of finite subgroups of $\mathrm{SL}(2,\mathbb C)$ over the polynomial algebra of two variables \cite{rv89,ch98}.
Inspired by this results, it is observed in \cite{gm02} that  the skew group algebras of a finite subgroup $G$ of $\mathrm{GL}(n+1,\mathbb C)$ over the exterior algebra of an $(n+1)$-dimensional vector space, and over polynomial algebra of $n$ variables, they have the  McKay quiver of $G$ as their Gabriel quivers.
The former of these algebras is an $n$-translation algebra and the later is a Noetherian Artin-Schelter regular algebra of Gelfand-Kirilov dimension $n+1$.
In this paper, we show that by constructing universal cover, taking $\tau$-slices, and then taking quadratic duals, we obtain tame $n$-slice algebras associate to the McKay quiver of a finite subgroup of $\mathrm{GL}(n+1,\mathbb C)$.
In this process, we get three pairs of quadratic dual algebras related to a McKay quiver.
We present some interesting equivalences of triangulate categories  related to these algebras, among them we get a version of Beilinson correspondence and a version of Bernstein-Gelfand-Gelfand correspondences related to McKay quivers.

We also describe the quivers and relations for the tame $2$-slice algebras related to McKay quivers.
We characterized the relations for $2$-slice algebras associated to the McKay quivers of finite Abelian subgroup of $\mathrm{SL}(3,\mathbb C)$ and of finite subgroup of $\mathrm{SL}(3,\mathbb C)$ obtained by embedding $\mathrm{SL}(2,\mathbb C)$ into $\mathrm{SL}(3,\mathbb C)$.
The relations of the $2$-translation algebras and AS-regular algebras related to the McKay quivers are described first, and the quivers and relations for the associated $2$-slice algebras follow immediately.

The paper is organized as follows.
In Section \ref{sec:pre}, concepts and results needed in this paper are recalled.
We recall the constructions of the $n$-translation quivers related to an $n$-properly-graded quiver, and the algebras related to an $n$-slice algebra.
We also recalled notions and results in \cite{gw00} needed in Section  \ref{sec:class}.
In Section \ref{sec:class}, we discuss the classification of $n$-slice algebras via the stable $n$-translation algebras and the $(n+1)$-preprojective algebras related to them.
In Section \ref{sec:mckayq}, we discuss the McKay quivers of some finite subgroups of $\mathrm{SL}(n+1,\mathbb C)$.
We also discuss the tame $n$-slice algebras associated to the McKay quivers of the finite subgroups in $\mathrm{SL}(n+1,\mathbb C)$ and present the equivalences of the related triangulate categories.
In the last section, we describe the relations for the  stable $2$-translation algebra, the AS-regular algebras and $2$-slice algebras related to the McKay quivers of finite Abelian subgroups of $\mathrm{SL}(3,\mathbb C)$ and finite subgroups of $\mathrm{SL}(3,\mathbb C)$ obtained by embedding $\mathrm{SL}(2,\mathbb C)$ into $\mathrm{SL}(3,\mathbb C)$.

\section{Preliminary}
\label{sec:pre}
In this paper, we assume that $k$ is a field, and $\LL = \LL_0 + \LL_1+\cdots$ is a graded algebra over $k$ with $\LL_0$ a direct sum of copies of $k$ such that $\LL$ is generated by $\LL_0$ and $\LL_1$.
Such algebra is determined by a bound quiver $Q= (Q_0,Q_1, \rho)$ \cite{g16}.

Recall that a bound quiver $Q= (Q_0,Q_1, \rho)$ is a quiver with $Q_0$ the set of vertices, $Q_1$ the set of arrows and $\rho$ a set of relations.
The arrow set $Q_1$ is usually defined with two maps $s, t$ from $Q_1$ to $Q_0$ to assign an arrow $\alpha$ its starting vertex $s(\alpha)$ and its ending vertex $t(\alpha)$.
Write $Q_t$ for the set of paths of length $t$ in the quiver $Q$, and write $kQ_t$ for space spanned by $Q_t$.
Let $kQ =\bigoplus_{t\ge 0} kQ_t$ be the path algebra of the quiver $Q$ over $k$.
We also write  $s(p)$ for the starting vertex of a path $p$ and $t(p)$ for the terminating vertex of $p$.
Write $s(x)=i$ if $x$ is a linear combination of paths starting at vertex $i$, and  write $t(x)=j$ if $x$ is a linear combination of paths ending at vertex $j$.
The relation set $\rho$ is a set of linear combinations of paths of length $\ge 2$. We may assume that the paths appearing in each of the element in $\rho$ have the same length, since we deal with graded algebra.
Through this paper, we assume that the relations are normalized such that each element in $\rho$ is a linear combination of paths of the same length starting at the same vertex and ending at the same vertex.
Conventionally, modules are assumed to be finitely generated left module in this paper.

Let $\LL_0=  \bigoplus\limits_{i\in Q_0} k_i$, with $k_i \simeq k$ as algebras, and let $e_i$ be the image of the identity in the $i$th copy of $k$ under the canonical embedding.
Then $\{e_i| i \in Q_0\}$ is a complete set of orthogonal primitive idempotents in $\LL_0$ and $\LL_1 = \bigoplus\limits_{i,j \in Q_0 }e_j \LL_1 e_i$ as $\LL_0 $-bimodules.
Fix a basis $Q_1^{ij}$ of $e_j \LL_1 e_i$ for any pair $i, j\in Q_0$, take the elements of $Q_1^{ij}$ as arrows from $i$ to $j$, and let $Q_1= \cup_{(i,j)\in Q_0\times Q_0} Q_1^{ ij}.$
Thus $\LL \simeq kQ/(\rho)$ as algebras for some relation set $\rho$, and $\LL_t \simeq kQ_t /((\rho)\cap kQ_t) $ as $\LL_0$-bimodules.
A path whose image is nonzero in $\LL $ is called a bound path.

A bound quiver $Q= (Q_0,Q_1,\rho)$ is called  {\em quadratic} if $\rho$ is a set of linear combination of paths of length $2$.
In this case, $\LL= kQ/(\rho)$ is called an {\em quadratic algebra}.
Regard the idempotents $e_i$ as the linear functions on $\bigoplus_{i\in Q_0} k e_i$ such that $e_i(e_j) = \delta_{i,j}$, and for $i,j\in Q_0$, arrows from $i$ to $j$ are also regarded as the linear functions on $e_j k Q_1 e_i$ such that for any arrows $\xa, \xb$, we have $\xa(\xb) =\delta_{\xa,\xb}$.
By defining $\xa_1 \cdots \xa_r (\xb_1, \cdots, \xb_r)=\xa_1(\xb_1) \cdots \xa_r ( \xb_r) $, $e_j kQ_r e_i$ is identified with its dual space for each $r$ and each pair $i,j$ of vertices, and the set of paths of length $r$ is regarded as the dual basis to itself in $e_j kQ_r e_i$.
Take a spanning set $ \rho^{\perp}_{i,j}$ of orthogonal subspace of the set $e_j k\rho e_i$ in the space $e_j kQ_2 e_i$ for each pair $i,j \in Q_0$, and set \eqqc{relationqd}{\rho^{\perp} = \bigcup_{i,j\in Q_0} \rho^{\perp}_{i,j}.}
The {\em quadratic dual quiver} of $Q$ is defined as the bound quiver $Q^{\perp} =(Q_0,Q_1, \rho^{\perp})$, and the {\em quadratic dual algebra} of $\LL$ is the algebra $\LL^{!,op} \simeq kQ/(\rho^{\perp})$ defined by the bound quiver $Q^{\perp}$ (see \cite{g20}).

\medskip

\subsection{$n$-properly-graded quiver and related $n$-translation quivers}
To define and study $n$-slice and $n$-slice algebra, we need  $n$-properly-graded quiver and related quivers.

Recall that a bound quiver $Q =(Q_0,Q_1,\rho)$  is called {\em $n$-properly-graded} if all the maximal bound paths have the same length $n$.
The graded algebra $\LL = k Q/(\rho)$ defined by an  $n$-properly-graded quiver is called an {\em $n$-properly-graded algebra.}

Let $Q$ be an $n$-properly-graded quiver with $n$-properly-graded algebra $\LL$, let $\caM = \caM_Q$ be a  maximal linearly independent set of maximal bound paths in $Q$.
Define {\em the returning arrow quiver} $\tQ= (\tQ_0,\tQ_1)$ with vertex set $\tQ_0 = Q_0$, arrow set $\tQ_1 = Q_1\cup Q_{1,\caM}$, with $Q_{1,\caM} = \{\beta_{p}: t(p) \to s(p)|  p\in \caM\}$.
By proposition 3.1 of \cite{g20}, $\tQ$ is the quiver of the trivial extension $\dtL$ of $\LL$.
So there is a relation set $\trho$, such that $\dtL \simeq k\tQ/(\trho)$.
We use the same notation $\tQ$ for the bound quiver $\tQ= (\tQ_0,\tQ_1,\trho)$.
By Lemma 3.2 of \cite{g20}, the relation set $\trho$ can be chosen as a union $ \rho \cup \rho_{\caM} \cup \rho_0$ of three sets, the relation set $\rho$ of $Q$, the set $\rho_{\caM}$ of all paths formed by two returning arrows from $Q_{1,\caM}$, and a set $\rho_0$ of linear combinations of paths containing exactly one returning arrow.

The $n$-translation quiver and $n$-translation algebra are introduced to study higher representation theory related to graded self-injective algebras in \cite{g16}.
The bound quivers of a graded self-injective algebras are exactly {\em the stable $n$-translation quivers}, with the Nakayama permutation $\tau$ as its $n$-translation \cite{g16} (called stable bound quiver of Loewy length $n+2$ in \cite{g12}).
Since $\dtL$ is a graded symmetric algebra, the returning arrow quiver $\tQ$ is a stable $n$-translation quiver with trivial $n$-translation.

With an $n$-properly-graded quiver $Q$, we can also define an acyclic stable $n$-translation quiver $\zzs{n-1} Q$, via its returning arrow quiver $\tQ$, as the bound quiver  with vertex set \eqqcn{vertexzq}{(\zZ|_{n-1} Q)_0 =\{(u , t)| u\in Q_0, t \in \zZ\}} and arrow set \eqqcn{arrowzq}{(\zZ|_{n -1}Q)_1  =  \zZ \times Q_1 \cup \zZ \times Q_{1,\caM} , } with $$ \zZ \times Q_1= \{(\alpha,t): (i,t)\longrightarrow (i',t) | \alpha:i\longrightarrow i' \in Q_1, t \in \zZ\},$$ and $$ \zZ \times Q_{1,\caM}= \{(\beta_p , t): (i', t) \longrightarrow (i, t+1) | p\in \caM, s(p)=i,t(p)=i'  \}.$$
The relation set of $\zZ|_{n-1} Q$ is defined as  \eqqcnn{relzq}{\rho_{\zZ|_{n-1} Q} =\zZ \rho\cup \zZ \rho_{\caM} \cup \zZ\rho_0,} where
\eqqcn{rela}{\zZ \rho =  \{\sum_{s} a_s (\xa_s,t)(\xa'_s,t) |\sum_{s} a_s \xa_s \xa'_s \in \rho, t\in \zZ\},}
\eqqcn{relb}{\zZ \rho_{\caM} =  \{(\beta_{p'},t)(\beta_p ,t+1)| \beta_{p'} \beta_{p}\in \rho_{\caM}, t\in \zZ\},} and \Small\eqqcn{relc}{\zZ \rho_0 = \{ \sum_{s'} a_{s'}  (\beta_{p'_{s'}}, t)  (\xa'_{s'}, t+1) +  \sum_{s} b_s (\xa_s,t)  (\beta_{p_s} ,t)|  \sum_{s'} a_{s'} \beta_{p'_{s'}} \xa'_{s'} + \sum_{s} b_s \xa_s \beta_{p_s} \in \rho_0 , t\in \zZ\},}\normalsize
when the returning arrow quiver $\tQ$ is quadratic.
The quiver $\zZ|_{n-1} Q$ is a  locally finite infinite quiver if $Q$ is locally finite.
It is a quiver with infinite copies of $Q$ with successive two neighbours connected by the returning arrows.
With the $n$-translation $\tau$ defined by sending each vertex $(i,t)$ to $(i,t-1)$, $\zzs{n-1}Q$ is a stable $n$-translation quiver.

Recall a {\em complete $\tau$-slice} in an acyclic stable $n$-translation quiver is a convex full bound sub-quiver which intersect each $\tau$-orbit exactly once.
It is obvious that each copy of $Q$ in $\zzs{n-1} Q$ is a complete $\tau$-slice.

Given a finite stable $n$-translation quiver $\tQ$, we can construct another infinite acyclic stable $n$-translation quiver $\zZ_{\vv} \tQ=(\zZ_{\vv} \tQ_0,\zZ_{\vv} \tQ_1, \rho_{\zZ_{\vv} \tQ}) $ as follows (denoted by $\olQ$ and called separated directed quiver in \cite{g12}):
the vertex set: \eqqcn{vertexzv}{\zZ_{\vv} \tQ_0=\{(i,n) | i \in Q_0, n \in \mathbb Z\};}
the arrow set: \eqqcn{arrowzv}{\zZ_{\vv} \tQ_1 = \{(\alpha, n):(i,n) \to (j,n+1)|\alpha: i \to j \in Q_1, n \in \mathbb Z \}.}
If $p = \alpha_s \cdots \alpha_1 $ is a path in $\tQ$, define $p[m] = (\alpha_s,m+s-1) \cdots (\alpha_1,m)$, and for an element $z=\sum_{t} a_tp_t$ with each $p_t$ a path in $\tQ$, $a_t \in k$, define $z[m]=\sum_{t} a_tp_t[m]$ for each $m \in \mathbb Z$.
Define relations
\eqqc{relationzv}{\rho_{\zZ_{\vv} \tQ} =\{\zeta[m]| \zeta \in \trho, m \in \mathbb Z \}}
here $\zeta[m] = \sum_{t} a_t p_t[m]$ for each $\zeta = \sum_{t} a_t p_t \in \rho $.
$\zZ_{\vv} \tQ$ is a locally finite bound quiver if $Q$ is so.

A connected quiver $Q$ is called {\em  nicely-graded}  if there is a map $d$ from $Q_0$ to $\zZ$ such that $d(j) = d(i) + 1$ for any arrow $\xa:i\to j$.

Clearly, a nicely graded quiver is acyclic.

\begin{pro}\label{zquivers}
\begin{enumerate}
\item Let $d$ be the greatest common divisor of the length of cycles in $\tQ$, then $\zZ_{\vv} \tQ$ has $d$ connected components.

\item All the connected components of $\zZ_{\vv} \tQ$ are isomorphic.

\item Each connected component of $\zZ_{\vv} \tQ$ is nicely graded quiver.
\end{enumerate}
\end{pro}
\begin{proof}
The first and the second assertions follow from Proposition 4.3 and Proposition 4.5 of \cite{g12}, respectively.
The last follows from the definition of $\zZ_{\vv} \tQ$.
\end{proof}

An $n$-properly-graded quiver is called {\em  $n$-nicely-graded quiver} if it is nicely-graded.

\begin{pro}\label{zquiversnicely}
Assume that $Q$ is an $n$-nicely-graded quiver.
Let $\tQ$ be the returning arrow quiver of $Q$.
Then $\zzs{n-1} Q$ is isomorphic to a connected component of $\zZ_{\vv} \tQ$, hence is also nicely graded.
\end{pro}
\begin{proof}
Let $d:Q_0 \to \zZ$ be the function such that $d(j) = d(i)+1$ for an arrow $\xa: i\to j$ in $Q$.
Fix an vertex $i_0\in Q_0$, define \eqqcnn{mapzz}{\phi: \zzs{n-1}Q_0 \to \zZ_{\vv}\tQ} by $$\phi(i,t) = (i,t(n+1)+d(i)-d(i_0)). $$
It is easy to see that $\phi$ is an isomorphism from $\zzs{n-1}Q$ to the connected component of $\zZ_{\vv}\tQ$ containing the vertex $(i_0,0)$.
\end{proof}

By definition, an $n$-translation quiver is an $(n+1)$-properly graded quiver.
The bound quiver $\zzs{n-1} Q$ is an $(n+1)$-nicely graded quiver and the $n$-translation sends the ending vertex of a maximal bound path of length $n+1$ to its starting vertex.
So by the definition of complete $\tau$-slice, we immediately have the following.
\begin{pro}\label{zquiversslice}
If $\tQ$ is a stable  $n$-translation quiver.
Then any complete $\tau$-slice in $\zZ_{\vv} \tQ$ is an $n$-nicely-graded quiver.
\end{pro}
Let $m\le l$ be two integers, write $\zZ_\vv Q [m,l]$ for the full subquiver of $\zZ_{\vv}\tQ$ with vertex set \eqqcn{}{\zZ_\vv Q_0 [m,l] = \{(i,t)| m \le t \le l\}.}
$\zZ_\vv Q [m,l]$ is a convex subquiver of $\zZ_{\vv} \tQ$, so it is regarded as a bounded subquiver by taking
\eqqc{relationzvml}{\rho_{\zZ_{\vv} \tQ}[m,l] =\{x| x \in \rho_{\zZ_{\vv} \tQ}, s(x),t(x) \in \zZ_\vv Q_0 [m,l] \}}
By Proposition 6.2 of \cite{g12}, we get the following for any $m$.
\begin{pro}\label{zqslice}
$\zZ_\vv Q [m,m+n]$ is a complete $\tau$-slice of $\zZ_{\vv} \tQ$.
\end{pro}

The complete $\tau$-slice $\zZ_\vv Q [0,n]$ is called an initial $\tau$-slice in \cite{g12}.
If the quiver $Q$ is nicely graded, by Lemma 6.5 of \cite{g12}, we can recover our quiver $Q$ by using $\tau$-mutations on a connected component of $\zZ_\vv Q [m,m+n]$.

Starting with an acyclic $n$-properly-graded quiver $Q$, we construct its returning arrow quiver $\tQ$, which is a stable $n$-translation quiver.
With $\tQ$, we construct a universal cover $\zzs{n-1} Q$, which is an acyclic stable $n$-translation quiver, or construct a universal cover $\zZ_{\vv} \tQ$, which is a nicely-graded stable $n$-translation quiver.
We can recover $Q$ as a complete $\tau$-slice of $\zzs{n-1} Q$, and when $Q$ is a nicely-graded, we can recover it as a connected component of complete a $\tau$-slice of $\zZ_{\vv} \tQ$.

We have the following picture illustrate the relationship among these quivers.

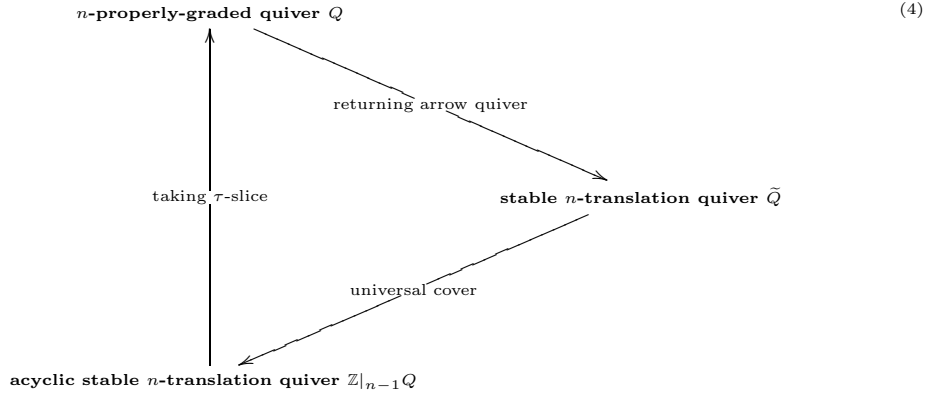
\tiny \eqqc{depictingqui}{\xymatrix@C=0.8cm@R=2.0cm{ \txt{\bf $n$-properly-graded quiver $ Q$}\ar@[black][dr]|-{\txt{returning arrow quiver }}&& \\
&\txt{ \color{black}{\bf stable $n$-translation quiver $\tQ $}}\ar@[black][dl]|-{\txt{universal cover}} \\
\txt{ \color{black}{\bf acyclic stable $n$-translation quiver $\mathbb Z|_{n-1} Q $}}\ar@[black][uu]|-{\txt{taking $\tau$-slice}}
}}\normalsize

\subsection{$n$-slice algebras and related algebras}
The quadratic dual quiver $Q^{\perp}=(Q_0,Q_1,\rho^{\perp})$ of an  acyclic quadratic $n$-properly-graded quiver $Q$ is called an {\em $n$-slice}.

Recall that a graded algebra $\tL= \sum_{t\ge 0} \tL_t$ is called a {\em $(p,q)$-Koszul algebra} if $\tL_t=0$ for $t> p$ and $\tL_0$ has a graded projective resolution \eqqc{projres_tL}{\cdots \lrw P^{t} \slrw{f_t}  \cdots \lrw P^{1}\slrw{f_1} P^{0} \slrw{f_0} \tL_0\lrw 0,} such that $P^{t}$ is generated by its degree $t$ part for $t\le q$ and $\ker f_q$ is concentrated in degree $q+p$ \cite{bbk02,g16}.
A graded algebra defined by an $n$-translation quiver $\tQ$ is called an {\em $n$-translation algebra}, if there is a $q \ge 2$ or $q=\infty$ such that $\tL$ is an $(n+1,q)$-Koszul algebra \cite{g16}.
A {\em stable $n$-translation algebra} $\tL$ is a $(n+1,q)$-Koszul self-injective algebra for some $q  \ge 2$ or $q=\infty$, and we call $q$ the {\em Coxeter index} of $\tL$.

Let $\LL$ be the $n$-properly-graded algebra defined by $Q$, if the trivial extension $\dtL$ is a stable $n$-translation algebra, the algebra $\GG=kQ/(\rho^\perp)$ defined by the $n$-slice $Q^{\perp}$ is called an {\em $n$-slice algebra}.

The following Proposition justify the name $n$-slice.
\begin{pro}\label{tau-slice}
Let $\GG$ be an acyclic $n$-slice algebra with bound quiver $Q^{\perp}$, then $Q$ is  a complete $\tau$-slice of the $n$-translation quiver $\zzs{n-1} Q$ and $\LL= \GG^{!,op}$ is a $\tau$-slice algebra.
\end{pro}

Starting with a quadratic acyclic $n$-properly-graded quiver $Q$, let $\LL$ be the $n$-properly-graded algebra defined by $Q$ and let $\GG$ be the algebra defined by the $n$-slice  $Q^{\perp}$.
If $\tQ$ is a quadratic quiver, write $\tQ^{\perp}$ for its quadratic dual quiver.
In this case, $\zzs{n-1} Q$ and $\zZ_{\vv} Q$ are also quadratic quiver, write $\zzs{n-1} Q^{\perp}$ and $\zZ_{\vv} Q^{\perp}$ for their quadratic dual quivers, respectively.

With the quivers related to $Q$ defined above and their quadratic duals, we have interesting algebras.
We associate the trivial extension $\tL=\dtL$ of $\LL$ to the returning arrow quiver $\tQ$.
If $\GG$ is an $n$-slice algebra, we associate the $(n+1)$-preprojective algebra $\tG= \Pi(\GG)$ of $\GG$ to  a twist $\tQ^{\perp,\nu}$ of its quadratic dual, by corollary 5.4 of \cite{g20}.

We also associate algebras  to the bound quivers $\zzs{n-1} Q$ and $\zZ_{\vv}\tQ$  by constructing smash products over $\tL$.
Note that $\tL$ is generated by the returning arrows over $\LL$.
Taking $\LL$ as degree zero component and taking the returning arrows as degree $1$ generators, we get a $\zZ$-grading on $\tL$ induced by the $n$-translation.
Construct the smash product $\tL \# k\zZ^*$ of $\tL$ with respect to this grading.
It is shown in Section 5 of  \cite{g16} that $\zzs{n-1} Q$ is the bound quiver of $\tL \# k\zZ^*$.

Taking the usual path grading on $\tL$, that is, taking $\LL_0$ as the degree $0$ component and taking the arrows in $\tQ$ as degree $1$ generators, we get a $\zZ$-grading on $\tL$ induced by the lengths of the paths of $k\tQ$.
Construct the smash product of $\tL \#' k\zZ^*$  with respect to this grading.
By Theorem 5.12 of \cite{g12}, $\zZ_{\vv}\tQ$ is the bound quiver of $\tL \#' k\zZ^*$.

So we can get  criteria for $\GG$ to be $n$-slice algebra using  $\tL \# k\zZ^*$ and $\tL \#' k\zZ^*$.
Note that $\tQ$ is always an $n$-translation quiver, so $\tL$ is $n$-translation algebra if and only if there is a $q \ge 2$ or $q=\infty$, such that $\tL$ is $(n+1,q)$-Koszul.
By Propositions 2.2 and  2.6 of \cite{g16}, $\tL$ is $(n+1,q)$-Koszul if and only if $\tL^{!,op}$ is $(q,n+1)$-Koszul.
We also get the following criterion for $\GG$ to be an $n$-slice algebra.

\begin{pro}\label{n-sliceprep}
Let $\GG$ be the algebra defined by an $n$-slice $Q^{\perp}$.
Then $\GG$ is an   $n$-slice algebra if and only if there is a $q \ge 2$ or $q=\infty$, such that the $(n+1)$-preprojective algebra $\Pi(\GG)$ is $(q,n+1)$-Koszul.
\end{pro}

Write $\zzs{n-1}Q^{\perp}$ and $\zZ_{\vv}\tQ^{\perp}$ for the quadratic duals of $\zzs{n-1}Q$ and $\zZ_{\vv}\tQ$, respectively.
Write $\wht{\LL}$ for the algebra defined by $\zzs{n-1}Q$ or $\zZ_{\vv}\tQ$, and write $\wht{\GG} $ for the algebra defined by $\zzs{n-1}Q^{\perp}$ or $\zZ_{\vv}\tQ^{\perp}$.
We have the following picture depicting the relationship of these algebras.
\eqqc{depictingalg}{\xymatrix@C=2.8cm@R1.5cm{ \txt{$\LL$} \ar@{<->}[rr]|-{\txt{ quadratic dual } }\ar@[black][dr]|-{\txt{trivial\\ extension}}&& \txt{ $\GG $} \ar@[black][dr]|-{\txt{ $(n+1)$-preprojective\\ algebra}} \\
&\txt{ \color{black}{$\tL $}}\ar@[black][dl]|-{\txt{smash product}} \ar@[black]@{<->}[rr]|-(.3){\txt{\color{black}{ quadratic dual }}} &&\txt{\color{black}{$\tG$} }\ar@[black][dl]\\
\txt{ \color{black}{$\ohL $}}\ar@[black][uu]|-{\txt{\color{black}{ $\tau$-slice algebra }}} \ar@[black]@{<->}[rr]|-{\txt{\color{black}{ quadratic dual }}} &&\ar@[black][uu]\txt{\color{black}{$\ohG$} }}
}
The left triangle is induced by the picture \eqref{depictingqui} depicting quivers.
The left vertically up arrow indicate taking a $\tau$-slice algebra when $\ohL = \tL\# k\zZ^*$, and indicate taking a direct summand of $\tau$-slice algebra when $\ohL = \tL\#' k\zZ^*$.

\medskip
\subsection{Loewy matrix and Koszul cone of stable $n$-translation algebra}
Now we assume that $\tL = \sum_{t=0}^{n+1} \tL_t$ is a stable $n$-translation algebra, that is, it is an $(n+1,q)$-Koszul  self-injective algebra, with finite bound quiver $\tQ$. Let $\tQ_0= \{1, 2, \ldots, m\}$.
Then $\tL$ is unital with $1= e_1+\cdots + e_m$.
We need some results in \cite{gw00} concerning the complexities of $\tL$ and the Gelfand-Kirilov of its quadratic dual $\tL^{!,op}$.

For $0\le t\le n+1$, define \eqqcn{Loewya}{A_t(\tL) =\mat{cccc}{\dmn_k e_1 \tL_t e_1 &\dmn_k e_2 \tL_t e_1& \cdots &\dmn_k e_m \tL_t e_1\\ \dmn_k e_1 \tL_t e_2 &\dmn_k e_2 \tL_t e_2& \cdots &\dmn_k e_m \tL_t e_2\\ \cdots&\cdots  & \cdots & \cdots  \\ \dmn_k e_1 \tL_t e_m &\dmn_k e_2 \tL_t e_m& \cdots &\dmn_k e_m \tL_t e_m }.}
Let $E$ be the $m \times m$ identity matrix, then $A_0(\tL) =E$.

The {\em Loewy matrix $L(\tL)$} for $\tL$ is defined in \cite{gw00} as a matrix
\eqqc{Loewym}{L(\tL)= \mat{cccccc}{A_1(\tL)& -E & 0 & \cdots & 0 & 0\\ A_2(\tL)&0 & -E & \cdots & 0 & 0\\ \cdot&\cdot & \cdot  & \cdots & \cdot & \cdot \\A_{n}(\tL)& 0 & 0 & \cdots & 0 &-E\\A_{n+1}(\tL)& 0 & 0 & \cdots & 0 &0}}
with size $(n+1)m \times (n+1)m$.

Let $M = M_{h} + \cdots + M_{h+n}$ be a graded $\tL$-module of Loewy length $\le n+1$ generated in degree $h$.
The {\em  level dimensional vector} $\ldmv M$ is defined as \eqqc{leveldmv}{\ldmv M = \mat{c}{\dmv M_{h} \\ \vdots \\ \dmv M_{h+n} },} where $\dmv M_t$ is the dimensional vector of $M_t $ as a $\tL_0$-module.
Recall the syzygy $\om{}M$ of $M$ is the kernel of the projective cover $P(M)\to M$.
Thus $\om{} M=0 $ if $M$ is projective, and there is an exact sequence $0\to \om{} M \to P(M) \to M\to 0$ if $M$ is not projective.
Write $\om{t+1}M = \om{}(\om{t} M)$.

Assume that $\tL$ is graded self-injective, if a finitely generated graded  $\tL$-module $M$ with out projective summands has a minimal projective resolution \eqqcn{}{\lrw P^s \lrw \cdots \lrw P^1\to P^0 \lrw M \to 0} such that $P^t$ is generated in degree $h+t$ for $t\le s$, then the level dimensional vector for $\om{s} M $ is defined and we have \eqqc{Lldim}{\ldmv \om{s} M = L^s(\tL) \ldmv M} by Proposition 1.1 of  \cite{gw00}.

Write $\bzero$ for the zero vector or zero matrix.
We call a matrix (vector) {\em nonnegative} if all the entries are $\ge 0$, {\em positive} if it is nonnegative and at least one entry is $>0$.
We call a matrix (vector) {\em negative} if its inverse is positive.
For two vectors (matrices) $\bv$ and $\bv'$, write $\bv \le \bv'$ if $\bv'-\bv$ is nonnegative and write $\bv < \bv'$ if $\bv'-\bv$ is positive.
Write $$\mathbf V_0 =\mat{c}{E\\ 0\\ \vdots \\0}_{m(n+1)\times m} .$$
The following Proposition follows from \eqref{Lldim} and the definition the $(p,q)$-Koszul algebra.

\begin{pro}\label{finitq}
Assume that $\tL$ is a stable $n$-translation algebra with Coxeter index $q$. Then $q$ is finite if and only if $L^{q+1}(\tL) \mathbf V_0$  is negative.
\end{pro}

Now assume that $q=\infty$, then $\tL$ is a Koszul self-injective algebra of Loewy length $n+2$.
Write $\caK$ for the real cone spanned by the level dimensional vectors of finitely generated Koszul $\tL$-modules, called Koszul cone in \cite{gw00}, $\caK$ is contained in the positive cone in the real space $\mathbb R^{(n+1)m}$.
Let $\varrho=\varrho_{L(\tL)}$ be the spectral radius of $L(\tL)$, that is, the maximal norm of its (complex) eigenvalues, call the size of the biggest Jordan block of the eigenvalue $\xl_i$ the algebraic degree of $\xl_i$.

We can restate the Theorems 2.4 and 2.5 of  \cite{gw00} as follows.
\begin{pro}\label{Koszulcone}  $\caK$ is a solid cone invariant under the linear transformation defined by $L(\tL)$.
\end{pro}
Then by Theorem 3.1 of \cite{v68}, $\varrho$ is an eigenvalue of $L(\tL)$ with maximal algebraic degree among the eigenvalues with norm $\varrho$.
By Theorem 2.7 of  \cite{gw00}, we have $\varrho\ge 1$.

The following is part (a) of Proposition 2.9 of \cite{gw00}.
\begin{pro}\label{Kcone}
Let $\tL$ be a  Koszul stable $n$-translation algebra.
Let $\varrho$ be the spectral radius of the Loewy matrix $L(\tL)$ with algebraic degree $d$.
Then for any vector $\bv \in \caK$, there exist an integer $1 \le d' \le d$, such that $L(\bv, d') = \{\frac{L^{r}(\tL) \bv}{r^{d'-1} \varrho^{r}}| r \in \mathbb N\}$ is a compact set in $\caK$ and there exists a sequence $r_1 < r_2 < \cdots $ of integers and a vector $\bv_0$ with nonnegative entries such that $\lim_{j \to \infty} \frac{L^{r_j}(\tL) \bv}{r_{j}^{d'-1} \varrho^{r_j}} =\bv_0$.
\end{pro}

Let $M$ be a $\tL$-module.
Recall that the {\em complexity} $c_{\tL}(M)$ of $M$ is defined as the least non-negative number $d$ such that there exists $\lambda >0$ with $\dmn_k \om{(t)} M \le \lambda \cdot t^{d-1}$ for almost all $t$, or infinite if no such $t$ exist.
The {\em complexity} $c(\tL)$ of an algebra $\tL$ is defined as the supremum of the complexities of finitely generated $\tL$-modules.

By applying Proposition \ref{Kcone}, the following is proven as Theorem 2.10 in \cite{gw00}.
\begin{pro}\label{ltone}
If $\varrho >1$, then $c_{\tL}(M)= \infty$ for any non-projective Koszul $\tL$-module $M$.
\end{pro}

Recall that for a graded algebra $\GG = \GG_0 +\GG_1 + \cdots $ generated by $\GG_0$ and $\GG_1$, {\em the Gelfand-Kirilov dimension $\gkdim \GG$} is defined  as \eqqc{gkdim}{\gkdim \GG = \varlimsup_{m\to \infty} \log_m \bigoplus_{t=1}^m \dmn_k \GG_t .}
By using Koszul duality, the Gelfand-Kirilov dimension of the quadratic dual $\tG = \tL^{!,op}$ is also discussed in \cite{gw00}.

The following  is Theorem 3.2 in \cite{gw00}.
Let $\tL$ be a Koszul stable $n$-translation algebra, and write  $\tG=\tL^{!,op}$.
\begin{thm}\label{GKdim_sptr}
If $\varrho=1$ and  $d$ is the algebraic degree  of the spectral radius, then $\gkdim \tG = d$.

If  $\varrho > 1$, then $\gkdim \tG = \infty$.
\end{thm}

The following  is Theorem 3.1 in \cite{gw00}.

\begin{thm}\label{Noe_sptr}
If $\tG$ is Noetherian, then $\varrho = 1$.
\end{thm}

\section{Classification of $n$-slice Algebras}
\label{sec:class}
By picture \eqref{depictingalg}, we see an $n$-slice algebra $\GG$ is related to a stable  $n$-translation algebra $\tL = (\Pi(\GG))^{!,op}$.

To classify the $n$-slice algebras, we first discuss a classification of the stable $n$-translation algebras, by expanding the theory developed in \cite{gw00}.
Let $ L(\tL)$ be the Loewy matrix of $\tL$, and $\varrho$ be the spectral radius of $L(\tL)$.

Recall that an algebra $\tL$ is called {\em periodic} if there is an integer $l$ such that $\om{l} \tL_0 \simeq \tL_0$ as $\tL$-module.

\begin{lem}\label{periodic}
A stable $n$-translation algebra $\tL$ is periodic if and only if its Coxeter index $q$ is finite.
\end{lem}
\begin{proof}
For stable $n$-translation algebra $\tL$ with finite $q$, the $(q+1)$th syzygy $\om{q+1}\tL_0 e_i$ of a simple $\tL$-module $\tL_0 e_i$ is a module concentrated in degree $n+1+q$, by the definition.
So $\om{q+1}\tL_0 e_i$ is semi-simple.
Since $\tL$ is self-injective, the syzygies of an indecomposable module are all indecomposable.
Thus $\om{q+1}\tL_0 e_i$ is simple, and action of $\om{q+1}$ permutes the simples.
This implies that there is a $l>0$ such that $\om{l(q+1)}$ acts trivially on the simples and $\tL$ is periodic.

Now assume that $\tL$ is periodic with minimal period $q'+1$, then by the definition $\om{q'+1}\tL_0 \simeq \tL_0$ with finite $q'$.
So it is semi-simple and hence concentrated in degree $n+1+q'$, since all the projective modules have the Loewy length $n+1$.
Since $\tL$ is $(n+1,q)$-Koszul, so $q'\ge 2$ and we see in the projective resolution \eqref{projres_tL} of $\tL_0$, $P^t$ is generated in degree $t$ for $t \le q'$ and $\ker f_q = \om{q'+1}\tL_0$ is concentrated in degree $n+1+q'$.
So $q=q'$ is finite.
\end{proof}

Clearly, being periodic is a special case of complexity $1$.

\medskip

We have the following characterization of the periodicity and of the complexities for stable $n$-translation algebras using Loewy matrices.
\begin{thm}\label{cplloewy}
Let $\tL$ be a stable $n$-translation algebra with Coxeter index $q$ and Loewy matrix $L(\tL)$.
\begin{enumerate}
\item The algebra $\tL$ is periodic if and only if there is an positive integer $h$, such that $L^h(\tL)\mathbf V_0$ is negative;

\item The algebra  $\tL$ is of finite complexity if and only if $q$ is infinite and the spectral radius of $L(\tL)$ is $1$;

\item The  algebra $\tL$ is of infinite complexity if and only if $q$ is infinite and the  spectral radius of $L(\tL)$ is larger than $1$.
\end{enumerate}
\end{thm}
\begin{proof}
Assume that $\tL$ is $(n+1,q)$-Koszul algebra.
By Lemma \ref{periodic}, $\tL$ is periodic if and only if $q$ is finite.
By Proposition \ref{finitq}, $q$ is finite if and only if $L^h(\tL) \mathbf v_0$ is negative for $h=q+1$.
This proves the first assertion.

Now assume that $q$ is infinite, then $\tL$ is Koszul.
By proposition \ref{Koszulcone}, $\caK$ is a solid cone invariant under the action of Loewy matrix $L(\tL)$.
By Theorem 3.1 of \cite{v68}, the spectral radius $\varrho$ is an eigenvalue of $L(\tL)$ with maximal algebraic degree among the eigenvalues with norm $\varrho$.
Note that the constant term of the characteristic polynomial of $L(\tL)$ is $1$ (see also Theorem 2.7 of \cite{gw00}).
So $\varrho \ge 1$.

If $\varrho >1$, by Proposition \ref{ltone},  $\tL$ is of infinite complexity.

Now assume that $\varrho=1$, then by Proposition \ref{Kcone}, there exist an integer $1 \le d' \le d$, such that $\{\frac{L^{r}(\tL) \ldmv M}{r^{d'-1} }| r \in \mathbb N\}$ is bounded set for any finitely generated Koszul $\tL$-module $M$.
So $\ldmv \om{r} M = L^{r}(\tL) \ldmv M \le r^{d-1} \bv$ for some positive vector $\bv$.
This implies that for sufficient large $r$, $\dmn_k \om{r} M \le \lambda r^{d-1}$ for some positive $\lambda$.
So the complexity of $M$ is $\le d$.
This proves that $\tL$ has finite complexity.
\end{proof}

Let $\tG$ be the quadratic dual of a stable $n$-translation algebra $\tL$, then $\tG$ is a $(q,n+1)$-algebra.
We can characterize the growth of $\tG$ as follows.

\begin{thm}\label{cls-dualGK}Let $\tL$ be a stable $n$-translation algebra with Coxeter index $q$ and spectral radius $\varrho$ and let $\tG=\tL^{!,op}$ be its quadratic dual.
Then
\begin{enumerate}
 \item $\tG$ is finite dimensional  if and only if $q$ is finite.

 \item $\tG$ is of finite Gelfand-Kirilov dimension if and only if $q$ is infinite and $\varrho = 1$.

 \item$\tG$ is of infinite Gelfand-Kirilov dimension if and only if $q$ is infinite and $\varrho > 1$.
\end{enumerate}
\end{thm}
\begin{proof}
If $q$ is finite, then $\tL$ is an almost Koszul algebra of type $(n+1,q)$, by the definition.
So $\tG$ is almost Koszul of type $(q,n+1)$ by Propositions 3.11 and 3.4 of \cite{bbk02}.
Thus it is of finite Loewy length, and hence is finite dimensional.
This proves the first assertion.

If $q$ is infinite, then $\tL$ is a Koszul self-injective algebra, and $\tG$ is an infinite dimensional AS-regular algebra.
The rest of the theorem follows from Theorem \ref{GKdim_sptr}.
\end{proof}

Combine  Theorems \ref{cplloewy} and \ref{cls-dualGK}, we get the following.
\begin{thm}\label{cls-dualpair}Let $\tL$ be a stable $n$-translation algebra and let $\tG=\tL^{!,op}$ be its quadratic dual.
Then
\begin{enumerate}
 \item $\tL$ is periodic if and only if $\tG$ is finite dimensional.

 \item $\tL$ is of finite complexity  if and only if $\tG$ is of finite Gelfand-Kirilov dimension.

 \item $\tL$ is of infinite complexity if and only if $\tG$ is of infinite Gelfand-Kirilov dimension.
\end{enumerate}
\end{thm}

By Theorem \ref{Noe_sptr}, we also have
\begin{thm}\label{Noe_GKdim}
If $\tG$ is Noetherian, then $\tG$ is of finite Gelfand-Kirilov dimension.
\end{thm}

Assume that $\GG$ is an acyclic $n$-slice algebra with quiver $Q^{\perp}$, let $\LL=\GG^{!,op}$ be its quadratic dual.
Let $\nu$ be the automorphism of $\LL$ defined by sending each arrow $\xa$ in $Q$ to $(-1)^n\xa$.
By Theorem 6.6 of \cite{g20}, we have that \eqqcn{}{\dtnL = \Delta_{\nu}\GG^{!,op}\simeq \Pi(\GG)^{!,op},}
where $\Pi(\GG)$ be the $(n+1)$-preprojective algebra,and $\dtnL$ be a twisted trivial extension of $\LL$ with respect $\nu$.
Let $\tG= \Pi(\GG)$  and  $\tL=\dtnL$.
$\tL$ is a stable $n$-translation algebra with trivial $n$-translation, and $\tG$ is an AS-regular algebra.
We have classification for $\tL$ and $\tG$, using numerical invariants such as complexity and Gelfand-Kirilov dimension, respectively.

When $n=1$, $\GG$ is an hereditary algebra, then $\GG=kQ$ is the path algebra of quiver $Q$.
It is well known that the path algebras are classified, according to their representation type, as finite, tame and wild ones.
This classification can be done using the underlying graphs of the quivers.
That is, a path algebra $kQ$ is of finite type if the underlying graph of $Q$ is a Dynkin diagram,  of tame type if the underlying graph of $Q$ is an Euclidean diagram and \ of wild type otherwise.
This classification can also be read out from the preprojective algebra $\Pi(Q)$ of the path algebra $kQ$.
The path algebra $kQ$ is of finite representation type if $\Pi(Q)$ is finite dimensional, of tame type if $\Pi(Q)$ is of Gelfand-Kirilov dimension $2$ and of wild type if $\Pi(Q)$ is of infinite Gelfand-Kirilov dimension, \cite{dr81}.

Now we define that an $n$-slice algebra $\GG$ is  of {\em finite type} if $\tG$ is finite dimensional, of {\em tame type} if  $\tG$ is of finite Gelfand-Kirilov dimension (not zero), and of {\em wild type}  if $\tG$ is of infinite Gelfand-Kirilov dimension.

As an immediate consequence of Theorem \ref{cls-dualpair}, we get a classification of $n$-slice algebras.
\begin{thm}\label{classification}
An $n$-slice algebra is  of finite type, or of tame type, or of wild type.
\end{thm}

For an $n$-slice algebra $\GG$,  we have the following characterizations of the classification.

\begin{thm}\label{cls_usen-tra}
\begin{enumerate}
 \item $\GG$ is of finite type if and only if $\tL$ is periodic, if and only if $L^h(\tL)\mathbf v_0$ is negative for some $h$, if and only if $q$ is finite.

 \item $\GG$ is of tame type if and only if $\tL$ is of finite complexity, if and only if $q$ is infinite and  the  spectral radius of $L(\tL)$ is $1$.

 \item $\GG$ is of wild type if and only if $\tL$ is of infinite complexity, if and only if  $q$ is infinite and  the  spectral radius of $L(\tL)$ is larger than $1$.
 \end{enumerate}
\end{thm}

\begin{proof}
This follows from the above Theorems \ref{cls-dualGK} and \ref{cls-dualpair}.
\end{proof}

%Note that by Lemma \ref{periodic}, $\tL$  is periodic if and only if $q$ is finite.
%\begin{pro}\label{cls_finite}
%$\GG$ is of finite type if and only if the Coxeter index $q$ is finite.
%\end{pro}

When $n=1$, $\tL $ is the algebra with vanishing radical cube and $q$ is related to  the Coxeter number of a Dynkin diagram \cite{bbk02}.
When $n=1$, we also have that, $\GG$ is tame if and only if $\tG$ is Noetherian,  by \cite{bgl87}.
For the general case, we have that the following Proposition, as a consequence of Theorem \ref{Noe_sptr}.
\begin{pro}\label{cls_tame}
If $\tG$ is Noetherian of infinite dimension, then $\GG$ is tame.
\end{pro}

It is natural to ask if the converse of Proposition \ref{cls_tame} is true?

\section{McKay Quivers and $n$-slice Algebras}
\label{sec:mckayq}

By Proposition \ref{cls_tame}, if the $(n+1)$-preprojective algebra $\tG$ of an $n$-slice algebra is Noetherian, then $\GG$ is tame.
In this Section, we discuss  $n$-slice algebras related to McKay quivers, they are always tame.

In this section, we assume that $k$ is an algebraically closed field of characteristic $0$.
It is well known that the quiver $\tQ$ of the preprojective algebra of a path algebra $kQ$ is the double quiver of $Q$.
The double quivers of Euclidean quivers are exactly the  McKay quivers of finite subgroup of $\mathrm{SL}(2,\mathbb C)$ \cite{m79}.

McKay quiver was introduced in \cite{m79}.
Let  $V$ be an $n$-dimensional vector space over $\mathbb C$ and let $G \subset \mathrm{GL}(n,\mathbb C) = \mathrm{GL}(V)$ be a finite subgroup.
$V$ is naturally a faithful representation of $G$.
Let $\{S_i | i =1, 2, \ldots, l\}$ be a complete set of irreducible representations of $G$ over $\mathbb C$.
For each $S_i$, decompose the tensor product $V \otimes S_i$ as a direct sum of irreducible representations, i.e., \eqqcn{fusion}{V \otimes S_i = \bigoplus_j S_j^{a_{i, j}}, \quad\mbox{  } \quad 1\le i\le l } here $ S_j^{a_{i, j}}$ denotes a direct sum of $a_{i,j}$ copies of $S_j$.
The McKay quiver $\tQ=\tQ(G)$ of $G$ is defined as follows.
The vertex set $\tQ_0 $ is the set of (indices of) the isomorphism classes of irreducible representations of $G$, and there are $a_{i,j}$ arrows from the vertex $i$ to the vertex $j$.
%%from Preliminary1

%new added2
%Little is known for McKay quivers in higher dimension.
We recall some results on the McKay quivers of Abelian groups and on the relationship between McKay quivers of same group in $\mathrm{GL}(n,\mathbb C)$ and in $\mathrm{SL}(n+1,\mathbb C)$ \cite{g09,g11}.

%%new added2

%from Preliminary2

Let $G$ be an  Abelian  group, then \eqqcn{}{G= G(r_1,\cdots,r_{m})= C_{r_1}\times \cdots \times C_{r_{m}}} is a direct product of $m$ cyclic groups of order $r_1,\ldots, r_{m}$ .
Write $\diag(x_1,\cdots,x_m)$ for the diagonal matrices with diagonal entries $x_1,\cdots,x_m$.
Let $\xi_{r}$ be the $r$th root of the unit, embed $G$ into $\mathrm{SL}(m+1,\mathbb C)$ by sending \eqqcn{abeliangpembed}{(\xi_{r_1}^{i_1},\xi_{r_2}^{i_2}, \cdots , \xi_{r_{m}}^{i_{m}}) \lrw \diag (\xi_{r_1}^{i_1},\xi_{r_2}^{i_2}, \cdots , \xi_{r_{m}}^{i_{m}}), }
for $\bi=(i_1, \cdots , i_m)\in \zZ/r_1\zZ \times \cdots \times \zZ/r_{m}\zZ$.
Let $\be_t = (0,\cdots,0,1,0\cdots,0) \in  \zZ/r_1\zZ \times \cdots \times \zZ/r_{m}\zZ$ be the element with $1$ of $\zZ/r_t\zZ$ at the $t$th position, and $0$ otherwise, and $\be= \sum_{t=1}^{m} \be_t$.
The following Proposition can be proven by using Proposition 3.2 of \cite{g11} and  Section 3 of \cite{g09}.
%new add3
We give an inductive proof to show how such quiver is constructed.
%%new add3

\begin{pro}\label{abelmcq}
The McKay quiver $\tQ(r_1,\cdots,r_{m})$  of the group $G(r_1,\cdots,r_{m})$ in $\mathrm{SL}(m+1,\mathbb C)$ is the quiver with  vertex set\eqqc{abelianqv}{\tQ_0(r_1,\cdots,r_{m}) = \zZ/r_1\zZ \times \cdots \times \zZ/r_{m}\zZ} and the arrow set \small\eqqc{abelianqa}{\arr{ll}{\tQ_1(r_1,\cdots,r_{m}) = & \{\xa_{\bi, t}: \bi  \to \bi +\be_t |\bi \in  \zZ/r_1\zZ \times \cdots \times \zZ/r_{m}\zZ, 1\le t \le m\} \\ & \qquad \qquad \cup \{\xa_{\bi,m+1}: \bi \to \bi -\be | \bi\in \zZ/r_1\zZ \times \cdots \times \zZ/r_{m}\zZ \}.}}\normalsize
\end{pro}
%new add4
\begin{proof}
We prove by induction on $m$.

The Abelian subgroup in $\mathrm{SL}(1,\mathbb C)$ is trivial and its  McKay quiver is a loop.
So Proposition holds for $m=0$.

Assume Proposition holds for $m=h$.

Embed the group $G(r_1,\cdots,r_h,r_{h+1})$ in $\mathrm{GL}(h+1,\mathbb C)$ by sending $$(\xrt{1},\cdots,\xrt{h}) \lrw \mathrm{diag}(\xrt{1}, \cdots, \xrt{h}, \xrt{h+1}),$$ then group $G(r_1,\cdots,r_h)$ is a subgroup of $G(r_1,\cdots,r_h,r_{h+1})$ with $$G(r_1,\cdots,r_h,r_{h+1})\cap \mathrm{SL}(h+1,\mathbb C) = G(r_1,\cdots,r_h),$$ and we have $$G(r_1,\cdots,r_h,r_{h+1}) = G(r_1,\cdots,r_h)\times C'_{r_{h+1}},$$ where $ C'_{r_{h+1}} = (\mathrm{diag}(1,\cdots,1,\xi_{r_{h+1}}))$ is the cyclic group generated by the element $\mathrm{diag}(1,\cdots,1,\xi_{r_{h+1}})$.
So $$G(r_1,\cdots,r_h,r_{h+1})/G(r_1,\cdots,r_h,r_{h+1})\cap \mathrm{SL}(h+1,\mathbb C) \simeq (\xi_{r_{h+1}}) .$$
By Theorem 1.2 of \cite{g11}, the McKay quiver of $G(r_1,\cdots,r_h,r_{h+1})$ in $\mathrm{GL}(h+1,\mathbb C)$ is a regular covering of the McKay quiver of $ G(r_1,\cdots,r_h)$ in $\mathrm{SL}(h+1,\mathbb C)$ with the automorphism group the cyclic group $(\xi_{r_{h+1}})$ of order $r_{h+1}$, whose generator is induced by the Nakayama permutation.
So the vertex set of the McKay quiver $\tQ'(r_1,\cdots,r_h,r_{h+1})$ for $G(r_1,\cdots,r_h,r_{h+1})$ in $\mathrm{GL}(h+1,\mathbb C)$ is $r_{h+1}$ copies of $\tQ(r_1,\cdots,r_h)$, indexed by $\zZ/r_{h+1}\zZ$.
Write $\bv\sk{h}$ for a vector $\bv$ to indicate that its  dimension is $h$.
Since the arrows given by the vector $\be\sk{h}$ represents the Nakayama transformation, it induces  arrows  $$\xa_{\bi,h+1}: (\bi\sk{h},t) \to (\bi\sk{h}-\be\sk{h}, t+1) $$ from one copy $\tQ(r_1,\cdots,r_h)$ to the next in $\tQ'(r_1,\cdots,r_h,r_{h+1})$.
The other arrows are determined by the representations of $G(r_1,\cdots,r_h) $, so the arrows given by the vectors $\be_t\sk{h}$ induce arrows inside each copy.
Changing the indices by subtracting $(\be\sk{h},0)$ from the indices of all the vertices, we may take
$$\arr{lll}{\tQ'_1(r_1,\cdots,r_h,r_{h+1}) &=& \{\xa_{\bi,t}: \bi \to \bi+\be_t\sk{h+1} \\&&\quad|\bi\in \bi \in  \zZ/r_1\zZ \times \cdots \times \zZ/r_{h+1}\zZ, 1\le t \le h+1 \}}$$ as the arrow set.

Now embed $G(r_1,\cdots,r_h,r_{h+1}) $ to $\mathrm{SL}(h+2,\mathbb C)$ by sending $$(\xi_{r_1}^{i_1}, \cdots , \xi_{r_{h+1}}^{i_{h+1}}) \lrw \diag (\xi_{r_1}^{i_1}, \cdots , \xi_{r_{h+1}}^{i_{h+1}}, \xrmt{1}\cdots\xrmt{h+1}) .$$
Since Nakayama permutation $\sigma$ of $\tQ'(r_1,\cdots,r_h,r_{h+1}) $ sends each vertex $\bi$ to the ending vertex of a path of length $h+1$ with arrows defined by different $\be_t$'s, that is, sending $\bi$ to $\bi-\be$.
Then by Theorem 3.1 of \cite{g11}, the McKay quiver $\tQ(r_1,\cdots,r_h,r_{h+1}) $ of $G(r_1,\cdots,r_h,r_{h+1}) $ in $\mathrm{SL}(h+2,\mathbb C)$ is obtained from the quiver $\tQ'(r_1,\cdots,r_h,r_{h+1}) $ by adding an arrow $\xa_{\bi,m+2}$ from $\bi$ to $\bi-\be$ for each vertex $\bi$ in $\tQ_0(r_1,\cdots,r_h,r_{h+1})$.

This shows that Proposition holds for $h+1$ and finishes the proof.
\end{proof}

Note that the Nakayama permutation for the subgroup of a special linear group is trivial.
%%new add4
As a direct consequence of Proposition 3.1 of \cite{g11}, we also have the following Proposition to describe the McKay quiver of finite group $G$ in $\mathrm{SL}(m+1,\mathbb C)$ obtained by embed $\mathrm{SL}(m,\mathbb C)$ into  $\mathrm{SL}(m+1,\mathbb C)$.
\begin{pro}\label{embedmcq}
Let $G$ be a finite subgroup of $\mathrm{SL}(m,\mathbb C)$ with McKay quiver $\tQ^{m}(G)$.
Then there is an embedding of $G$ into  $\mathrm{SL}(m+1,\mathbb C)$ such that the McKay quiver $\tQ^{m+1}(G)$ of $G$, as a subgroup of $\mathrm{SL}(m+1,\mathbb C)$, is obtained from $\tQ^{m}(G)$ by adding a loop at each vertex of $\tQ^{m}(G)$.
\end{pro}

Let $V$ be an $(n+1)$ dimensional vector space over $k$.
let $k[V]= k[x_0, x_1, \cdots, x_n]$ be the polynomial algebra of $n+1$ variables over $k$ and $ \wedge V$   be exterior algebra of $V$.
In fact, both  $k[V]$ and $ \wedge V$ are quotient algebras of the tensor algebra $T_k(V)$ of $V$ over quadratic ideals, and they are quadratic dual each other.
Let $G$ be a finite subgroup  of $\mathrm{GL}(V) \simeq \mathrm{GL}(n+1,\mathbb C)$.
The following is the Theorem 1.8 of \cite{gm02}.
\begin{thm}\label{skewMcKay}
Let $\tQ$ be the McKay quiver of $G$ in $\mathrm{GL}(V)$, the skew group algebra $T_k(V)*G$ is Morita equivalent to $k\tQ$.
\end{thm}

Let $k[V]*G$ and $ \wedge V *G$ be the skew group algebras of $G$ over $k[V]$ and over $ \wedge V$, respectively.
Let $\tG = \tG(G)$ be the basic algebra Morita equivalent to $ k[V]*G$ and let $\tL= \tL(G) $ be the basic algebra Morita equivalent to $ \wedge V*G$.
By Theorem \ref{skewMcKay}, the quivers of both $\tG$ and $\tL$ are the McKay quiver $\tQ=\tQ(G)$ of $G$ in $\mathrm{GL}(n+1,\mathbb C)$.

Note that $ \wedge V$ is the Yoneda algebra of $k[V]$,  and both are Koszul algebra.
By Theorem 10 of \cite{m01}, $ \wedge V *G$  is the Yoneda algebra of $k[V]*G$, and by Theorem 14  of \cite{m01}, they are both Koszul.
This shows that $\tG$ and $\tL$ are Yoneda algebras from each other, so they are quadratic dual of each other. By Lemma 13 of \cite{m01}, $ k[V] *G$ is an AS-regular algebra(called generalized Auslander algebra there) and $ \wedge V *G$ is a self-injective algebra.
By choosing a quadratic relation $\trho= \trho(G)$ of $\tQ(G)$ such that $\tL \simeq  k\tQ/(\trho)$.
The following Proposition follows from the definition of $n$-translation algebra.
\begin{pro}\label{mckpair}
Let $G$ be the finite subgroup of $\mathrm{GL}(n+1,\mathbb C)$, then its McKay quiver  $\tQ(G) = (\tQ_0,\tQ_1,\trho)$ is an $n$-translation quiver and $\tL(G)$ is a Koszul stable $n$-translation algebra.
\end{pro}

Starting with a finite group  $G\subseteq \mathrm{GL}(n+1,\mathbb C)$, we have a pair of bound quivers $\tQ(G)$ and $\tQ^{\perp}(G)$ associates to the McKay quiver of $G$.
The bound quivers $\tQ(G)$ defines a stable $n$-translation algebra $\tL(G)$ and $\tQ^{\perp}(G)$ defines an AS-regular algebra $\tG(G)=k\tQ/(\rho^{\perp})$.
We call $\tL(G)$ a {\em stable $n$-translation algebra associated to $\tQ(G)$} and $\tG(G)$ an {\em AS-regular algebra associated to $\tQ(G)$}.

Let $Q^N(G)=\zZ_{\vv}\tQ[0,n]$, then by Theorem 5.1 of \cite{g12}, the algebra $\LL^N$ defined by the bound quiver $Q^N$ is the Beilinson algebra of $\tL(G)$.
Let $\tL(G)\#'k\zZ^*$ be the smash product of $\tL(G)$ with respect to the usual gradation (see \cite{g16}).
Write $\widehat{\LL}$ for the repetitive algebra of an algebra $\LL$ of finite global dimension, then by Theorem 5.12 of \cite{g12}, we have $\tL(G)\#'k\zZ^* \simeq \widehat{\LL^N}$.

Let $Q = Q(G)$ be a connected component of a complete $\tau$-slice of $\zZ_{\vv} \tQ$, then $Q$ is a quadratic $n$-nicely-graded quiver.
Let $Q^{\perp}$ be its quadratic dual quiver, $Q^{\perp}$ is a nicely-graded $n$-slice.
We have that $\zzs{n-1}Q(G)$ is a connected component of $\zZ_{\vv} \tQ(G)$, and use the notations of the vertices and arrows in $\zZ_{\vv} \tQ(G)$ for the vertices and arrows in $\zzs{n-1}Q(G)$.
Let $Q'$ be the full subquiver of $\zzs{n-1}Q(G)$ defined by the vertex set $\{(i,m)\in (\zzs{n-1}Q(G))_0 |0\le m \le n \}$.
Then $Q'$ is a complete $\tau$-slice in $\zzs{n-1}Q(G)$ and $Q$ is obtained by a sequences of $\tau$-mutations from $Q'$.
Let $\LL'$ be the $\tau$-slice algebra defined by $Q'$, then by Theorem 6.8, $\Delta(\LL') \simeq \Delta(\LL) =\Delta(G)$.
%So the bound quiver of $\Delta(G)$ is exactly the returning arrow quiver $\tQ'$.
So the bound quiver of $\Delta(G)$ is the returning arrow quiver $\tQ'$ of $Q'$ whose vertex set is indexed as $$\{(i,\bar{m})| (i,m)\in \zzs{n-1}Q(G), \bar{m} \in \zZ/(n+1)\zZ\}.$$
For a path $p= \xa_l\cdots \xa_l \in \tQ_l(G)$ and $\bam\in \zZ/(n+1)\zZ $, write $p[\bam]$ for the path $ (\xa_l,\overline{m+l-1}) \cdots (\xa_1,\bam)$ of length $l$ in $\Delta(G)$.
Then we have that $\sum_{p,\bam} a_{p,\bam} p[\bam] =0$ if and only if $\sum_{p,\bam=\bam'} a_{p,\bam} p =0$ for all $\bam'\in \zZ/(n+1)\zZ$.

\begin{pro}\label{mcknngtm}
$\GG(G)= kQ^{\perp}/(\rho^{\perp})$ is a tame $n$-slice algebra.
\end{pro}

\begin{proof}
For any $i\in \tQ(G)$, we have a minimal projective resolution
\eqqc{projtLGi}{\slrw{} \tP^t \slrw{f_t} \tP^{t-1}  \slrw{f_{t-1}} \cdots  \tP^1 \slrw{f_1} \tP^{0} \lrw \tL_0(G) e_i \lrw 0 }
for the simple $\tL(G)$-module $\tL_0(G) e_i$, such $\tP^t$ is generated in degree $t$.
We may assume that $\tP^t = \bigoplus_{s=1}^{h_t} \tL(G) e_{i_{t,s}}$ for a sequence $I_t =(i_{t,1}, \ldots, i_{t,h_t})$ of vertices in $\tQ_0(G)$.
Then $f_t$ is presented as a matrix $C_t=(c_{u,v}\sk{t})$ of size $h_{t}\times h_{t-1}$.
We have that elements $c_{u,v}\sk{t}$ of $C_t$ are all in $\sum_{s\ge 1}\tL_s(G)$ for $1\le u \le h_{t-1}, 1\le v \le h_t$.

Clearly \eqref{projtLGi} is exact if and only if $C_{t}C_{t+1} =0$ and if $(d_1,\cdots, d_{h_{t}})C_t =0$ for some $(d_1,\cdots, d_{h_{t}}) \in \bigoplus_{s=1}^{h_{t}} \tL(G) e_{i_{t,s}}$, then there is $(d'_1, \cdots, d'_{h_{t+1}})\in \bigoplus_{s=1}^{h_{t+1}} \tL(G) e_{i_{t+1,s}}$ such that \eqqcn{}{(d_1,\cdots, d_{h_{t}}) =(d'_1, \cdots, d'_{h_{t+1}})C_{t+1} }

Now for any $(i,\bam)\in \Delta(G)$, define $\tP^t[\bam] =  \bigoplus_{s=1}^{h_t} \Delta(G) e_{i_{t,s},\overline{m+t}}$, and let $\tilde{f_t}[\bam]$ be the map defined by the $h_{t}\times h_{t-1}$ matrix $C_t[\bam] = (c_{u,v}\sk{t}[\overline{m+t}])$.
Then we get immediately a sequence
\eqqc{projtLGizvvm}{\cdots \slrw{} \tP^t[\bam] \slrw{\tilde{f}_t[\bam]} \tP^{t-1}[\bam]  \slrw{\tilde{f}_{t-1}[\bam]} \cdots  \tP^0[\bam] \slrw{\tilde{f}_1[\bam]} \tP^{0} \lrw \Delta(G) e_{i,\bam} \lrw 0.}
By comparing the matrices defining the sequences \eqref{projtLGi} and \eqref{projtLGizvvm}, we see that  \eqref{projtLGi} is exact if and only if \eqref{projtLGizvvm} is exact.
That is, \eqref{projtLGi} is a projective resolution of simple $\tL(G)$-module $\tL_0(G) e_i$ if and only if  \eqref{projtLGizvvm} is a projective resolution of simple $\Delta(G)$-module $\Delta_0(G) e_{i,\bam}$ for any $\bam \in \zZ/(n+1)\zZ$.

Since $\tL(G)$ is Koszul, $P^t $ is generated in degree $t$, so $c_{u,v}\sk{t} \in \tL_1(G)$ and hence $c_{u,v}\sk{t}[\overline{m+t}] \in \Delta_1(G)$ and $\tP^t[\bam] $ is generated in degree $t$, for all $(i,\bam)$.
This shows that $\Delta(G)$ is Koszul and $\tG(G)$ is an $n$-slice algebra.

By comparing \eqref{projtLGi} and \eqref{projtLGizvvm}, the simple modules $\Delta_0(G) e_{i,\bam}$ has the same complexity as $\tL_0(G) e_i$.
By Theorem \ref{Noe_sptr} and Theorem \ref{cplloewy}, $\tL(G)$ is of finite complexity,  so  $\Delta(G)$ has finite complexity.
So by Theorem \ref{cls_usen-tra}, $\GG(G)$ is a tame $n$-slice algebra.
\end{proof}

We call $\GG(G)$ an {\em $n$-slice algebra associated to the McKay quiver $\tQ(G)$}.

It is interesting to know if the converse of Proposition \ref{mcknngtm} is true, that is, for an indecomposable nicely graded tame $n$-slice algebra $\GG$, if there is a finite subgroup $G \subset \mathrm{GL}(n+1,\mathbb C)$, such that $\GG\simeq \GG(G)$ for some connected component $Q$ of a complete $\tau$-slice in $\zZ_{\vv} \tQ(G)$?
The quiver $Q(G)$ is not unique by the definition, but they can be related by a sequences of $\tau$-mutations.
Thus $n$-slice algebra $\GG(G)$ associated to the McKay quiver of $G$ is not unique, and such algebras are related by a sequence of $n$-APR tilts (see \cite{gx20}).
Let $\GG^N = \LL^{N,!,op}$ be the quadratic dual of $\LL^N$.
We remark that $\GG(G)$ can be obtained from a direct summand of $\GG^N$ by a sequence of $n$-APR tilts.

\medskip

For the three pairs of quadratic duals of Algebras: $(\tL(G),\tG(G))$, $(\widehat{\LL^N},\widehat{\GG^N})$ and $(\LL^N, \GG^N)$.
We can construct $\widehat{\LL^N}$ from $\tL(G)$ by smash product, and $\LL^N$ from $\widehat{\LL^N}$ by taking $\tau$-slice algebra, we don't know how to get back $\tL(G)$ from $\LL(G)$.

For an AS-regular algebra $\tG$, denote by $\mathrm{qgr} \tG$ the non-commutative projective scheme of $\tG$ (see \cite{az94}).
We have the following equivalences of triangulate categories related to McKay quiver.

\begin{thm}\label{mckequi}
The following categories are equivalent as triangulate categories:
\begin{enumerate}
\item  the bounded derived category $\mathcal D^b( \mathrm{qgr} \tG(G)) $ of the non-commutative projective scheme of $\tG(G)$,
\item  the bounded derived category $\mathcal D^b(\GG^N) $ of the finitely generated $\GG^N$-modules,
\item  the bounded derived category $\mathcal D^b(\LL^N) $ of the finitely generated $\LL^N$-modules,
\item  the stable category $\underline{\mathrm{grmod}} \tL(G)$ of finitely generated graded $\tL(G)$-modules,
\item the stable category $\underline{\mathrm{mod}} \widehat{\LL^N}  $ of finitely generated $\widehat{\LL^N} $-modules,

\item the stable category $\underline{\mathrm{mod}} \widehat{\GG^N}  $ of finitely generated $\widehat{\GG^N} $-modules.
\end{enumerate}
If the lengths of the oriented circles in $\tQ(G)$ is coprime, they are also equivalent to the following triangulate categories.
\begin{enumerate}
\item[(7)] the bounded derived category $\mathcal D^b( \mathrm{qgr} \Pi(G)) $ of the non-commutative projective scheme of the $(n+1)$-preprojective algebra $\Pi(G)$ of $\GG(G)$,

\item[(8)]  the stable category $\underline{\mathrm{grmod}} \Delta_{\nu}(G)$ of finitely generated graded modules over the (twisted) trivial extension $\Delta_{\nu}(G)$ of $\LL(G)$.
\end{enumerate}
\end{thm}
\begin{proof}
By Theorem 4.14 of \cite{mm11}, we have that $\mathcal D^b( \mathrm{qgr} \tG(G)) $ is equivalent to $\mathcal D^b(\GG^N) $ as triangulate categories, since $\tG(G)$ is AS-regular.

By Theorem 1.1 of \cite{c09}, we have that ${\mathrm{grmod}} \tL(G)$ and ${\mathrm{grmod}} \Delta (\LL^N)$ are equivalent, thus $\underline{\mathrm{grmod}} \tL(G)$ and $\underline{\mathrm{grmod}} \Delta (\LL^N)$ are equivalent as triangulated categories.

By Corollary 1.2 of \cite{c09}, we have that $\underline{\mathrm{grmod}} \tL(G)$ and $\caDb(\mmod \LL^N)$ are equivalent as triangulated categories.

On the other hand, we have $\GG^N \simeq \LL^{N,!,op}$ is Koszul by Propositions 2.6 and 6,5 of \cite{g20}.
By Proposition 2.5,  and Corollary 2.4 of \cite{g20}, we have that $\caD^b(\mmod \GG^N) $ and $ \caD^b(\mmod \LL^N)$ are equivalent to the orbit categories of $\caD^b(\grm \GG^N)$ and of $\caD^b(\grm \LL^N)$, respectively, using the proof Theorem 2.12.1 of \cite{bgs96} (see the arguments before Theorem 6.7 of \cite{g20}).
So by Theorem 2.12.6 of \cite{bgs96}, $\caDb(\mmod \GG^N)$ and $\caDb(\mmod \LL^{N})$ are equivalent as triangulated categories.
Similar to the proof of  Theorem 6.7 of \cite{g20}, we also get that $\mathcal D^b( \mathrm{qgr} \tG(G)) $ and $\underline{\mathrm{grmod}} \tL(G)$ are equivalent as triangulated categories.

By Lemma II.2.4 of \cite{h88}, $\mmod \widehat{\LL^N}$ and $\mathrm{grmod}\Delta(\LL^N)$ are equivalent, so $\underline{\mmod} \widehat{\LL^N}$ and $\underline{\mathrm{grmod}} \Delta(\LL^N)$ are equivalent as triangulated categories.
By Theorem II.4.9 of \cite{h88}, $\caDb(\mmod\LL^N) $ and $\underline{\mmod} \widehat{\LL^N}$ are equivalent as triangulated categories.

Similarly, $\caDb(\mmod\GG^N) $ and $\underline{\mmod} \widehat{\GG^N}$ are equivalent as triangulated categories.

When the lengths of the oriented circles in $\tQ(G)$ is coprime, then $\zZ_{\vv}\tQ(G)$ has only one connected component, by Proposition \ref{zquivers}.
So $\LL^N = \LL(G)$ and equivalences for (2),(3),(7) and (8) follow from Theorem 6.7 of \cite{g20}.
\end{proof}

We remark that the equivalence of (1) and (2) can be regarded as a McKay quiver version of Beilinson correspondence and the equivalence of (1) and (4) can be regarded as a McKay quiver version of Berstein-Gelfand-Gelfand correspondence \cite{b78,bgg78}.
So we have the following analog to \eqref{depictingalg}, for the equivalences of the triangulate categories in Theorem \ref{mckequi}.

\tiny\eqqc{depictingcat}{\xymatrix@C=2.0cm@R1.5cm{ \txt{$\caDb(\LL^N$)} \ar@{<->}[rr]^-{\txt{} }\ar@{<->}@[black][dr]|-{\txt{}}&& \txt{ $\caDb(\GG^N) $} \ar@{<->}@[black][dr]|-{\txt{ Beilinson equivalence }} \\
&\txt{ \color{black}{$\underline{\mathrm{grmod}} \tL(G) $}}\ar@{<->}@[black][dl]|-{\txt{}} \ar@[black]@{<->}[rr]|-{\txt{\color{black}{BGG equivalence}}} &&\txt{\color{black}{$\mathcal D^b( \mathrm{qgr} \tG(G)) $} }\ar@{<->}@[black][dl]\\
\txt{ \color{black}{$\underline{\mmod} \widehat{\LL^N} $}} \ar@{<->}@[black][uu] \ar@[black]@{<->}[rr]^-{\txt{\color{black}{}}} &&\ar@{<->}@[black][uu]\txt{\color{black}{$\underline{\mmod} \widehat{\GG^N}$} }}
}\normalsize

\section{$2$-slice Algebras Associated to McKay Quivers}
\label{sec:twoslicea}
In the classical representation theory, we take a slice from the translation quiver and view the path algebra as $1$-slice algebra.
For $n\ge 2$, we need to know both quiver and the relations for a $n$-slice algebra, especially for the tame ones.
For any $n$, we can construct tame $n$-slice algebras using the McKay quivers.

Now we consider the case of $n=2$, by writing down their quivers and relations.
We first determine the relations which produce stable $2$-translation algebra of finite complexity for a given McKay quiver $\tQ$.
By constructing the bound quiver $\zZ_{\vv}\tQ$ and taking a complete $\tau$-slice, then taking the quadratic dual, we get the related $2$-slice quiver and $2$-slice  algebra.

Assume that $\tQ= \tQ(G)= (\tQ_0,\tQ_1)$ is a McKay quiver of finite subgroup $G$ of $\mathrm{SL}(3,\mathbb C)$.
Let $\tL=k\tQ/I$ be a stable $2$-translation algebra of finite complexity.
Note that $\tL = \sum_{t=1}^3\tL_t $ is a stable $2$-translation algebra with trivial $2$-translation.
Now we want to determine the relation $\trho$ such that $I = (\rho)$.

Since $\tL$ is Koszul, $\trho \subset k \tQ_2$,  the $k$-space spanned by the paths of length $2$.
By Lemma 2.1 of \cite{gw00}, we have $\dmn_k e_i \tL_1 e_j = \dmn_k e_j \tL_2 e_i$ and $\tL_3e_i \simeq \tL_0 e_i$ as $\tL_0$ module.
Let $M(\tQ)$ be the adjacent matrix of $\tQ$, that is, an $|\tQ_0|\times|\tQ_0|$ matrix with the $(i,j)$ entry the number of arrows from $i$ to $j$, then the Loewy matrix of $\tL$ is \eqqc{LoewyL}{ L(\tL) = \mat{cccc}{ M(\tQ)  &-E &0\\  M'(\tQ)& 0& -E\\ E& 0& 0}.}
This is exactly the Loewy matrix of $\tL(G)$, the basic algebra Morita equivalent to $\wedge[x_0,x_1,x_2]*G$.
Therefore we have the following Proposition.
\begin{pro}\label{commrel} Let $i,j \in \tQ_0$.
Then
\begin{enumerate}
\item If there is an arrow from $i$ to $j$, then there is a bound  path of length $2$ from $j$ to $i$ in $\tL$.

\item If there is only one arrow from $i$ to $j$, any two paths of length 2 from $j$ to $i$ are linearly dependent in $\tL$.
\end{enumerate}
\end{pro}

\begin{proof}
The proposition follows directly from  $\dmn_k e_i \tL_1 e_j = \dmn_k e_j \tL_2 e_i$.
\end{proof}
Therefore the number of arrows from $i$ to $j$ is exactly the order of the set of maximal linearly independent sets of bound paths of length $2$ from $j$ to $i$.

We also have the following Proposition
\begin{pro}\label{zerorel}
If $\xa:i\to j,\xb:j\to h$ are arrows in $\tQ$ with $i,j,h$ different vertices in $\tQ$ and there is no arrow from $h$ to $i$, then $\xb\xa =0 $ in $\tL$.
\end{pro}
\begin{proof}
Due to that there is no arrow from $h$ to $i$, $\dmn_k e_h\tL_1 e_i =0$.
Since $\tQ$ is a $2$-translation quiver with $\tau i =i$, thus by definition, $\dmn_k e_i\tL_2 e_h = \dmn_k e_h\tL_1 e_i =0$.
Thus for any arrows $\xa:i\to j,\xb:j\to h$, we have $\xb\xa =0$ in $\tL$.
\end{proof}

\medskip
Now we determine the relations for the McKay quiver of finite Abelian subgroup of $\mathrm{SL} (3, \mathbb C)$ and of the subgroup which is a finite subgroup of $\mathrm{SL}(2, \mathbb C)$ embedded in $\mathrm{SL} (3, \mathbb C)$.
Due to Proposition \ref{zerorel}, we excluded the McKay quivers of Abelian group which has a direct summand of $\zZ/3\zZ$.

Let $s \ge  4, r\ge 4$ and let  $G(s,r)$ be the finite subgroup of $\mathrm{SL}(3, \mathbb C)$ obtained from  the Abelian groups $\zZ/s\zZ\times \zZ/r\zZ $, using the embedding in Proposition \ref{abelmcq}.
Write $\Xi$ for $A_l, D_l, E_6, E_7$ and $E_8$ with $l\ge 5$.
Let  $G(\Xi)$ be the finite subgroup of $\mathrm{SL} (3, \mathbb C)$ obtained from finite subgroup $G(\Xi)$ of $\mathrm{SL}(2, \mathbb C)$ with McKay quiver of type $\Xi$, using the embedding in Proposition \ref{embedmcq}.
Let $\tQ(G)$ be the McKay quiver of subgroup $G$, now we determine the relations for the stable $2$-translation  algebras with McKay quiver  $\tQ(G)$ as their quivers and for the $2$-slice algebras defined by the quiver $\tQ(G)$.
Denote by $k^*$ the set of nonzero elements in $k$.
Write the McKay quiver for $G(s,r)$ as $\tQ(s,r)$ and the one for $G(\Xi)$ as $\tQ(\Xi)$.

\subsection{Relations for $\tQ(s,r)$}:
Let $\tQ = \tQ(s,r)$ be the McKay quiver of an Abelian group which is a direct sum of cyclic group of order $s$ and $r$.
By Proposition \ref{abelmcq}, its vertex set $\tQ_0 = \zZ/s \zZ \times \zZ/r \zZ$, and the arrows are described as follows: for each vertex $\bi \in \tQ_0$, there is an arrow $\xa_\bi$ from $\bi$ to $\bi+\be_1$, an arrow $\xb_\bi$ from $\bi$ to $\bi+\be_2$ and an arrow $\xc_{\bi}$ from $\bi$ to $\bi-\be$.
So the McKay quiver $\tQ(s,r)$ is as follows:
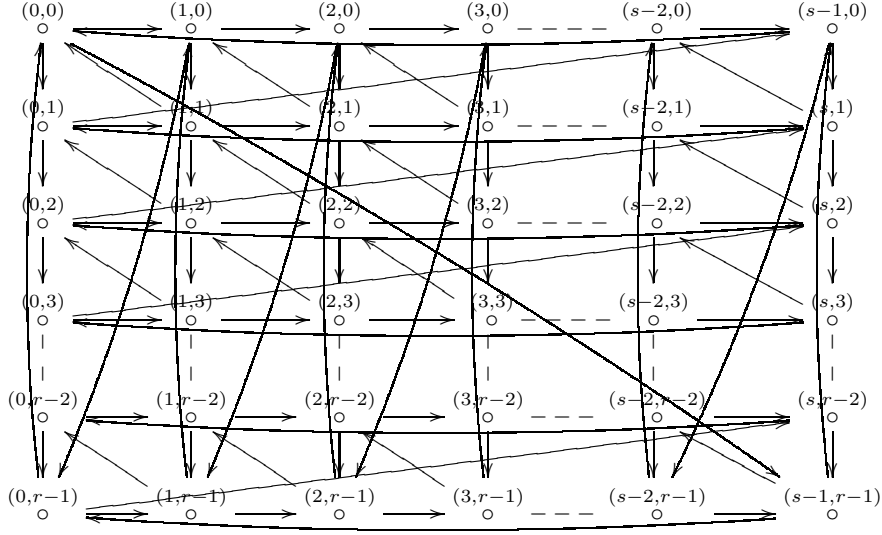
\eqqcn{McKayquiverGsr}{
\xymatrix@C=0.8cm@R0.6cm{
\stackrel{(0,0)}{\circ} \ar[r]\ar@[green][d]\ar@[red]@/^/[dddddrrrrr] & \stackrel{(1,0)}{\circ} \ar[r]\ar@[green][d]\ar@[red]@/^/[dddddl] & \stackrel{(2,0)}{\circ} \ar[r]\ar@[green][d]\ar@[red]@/^/[dddddl] & \stackrel{(3,0)}{\circ} \ar@{--}[r]\ar@[green][d]\ar@[red]@/^/[dddddl] & {} \stackrel{(s-2,0)}{\circ} \ar[r]\ar@[green][d] & \stackrel{(s-1,0)}{\circ} \ar@/^/[lllll]\ar@[green][d]\ar@[red]@/^/[dddddl] \\
\stackrel{(0,1)}{\circ} \ar[r]\ar@[green][d]\ar@[red][urrrrr] & \stackrel{(1,1)}{\circ} \ar[r]\ar@[green][d]\ar@[red][ul]  & \stackrel{(2,1)}{\circ} \ar[r]\ar@[green][d]\ar@[red][ul]  & \stackrel{(3,1)}{\circ} \ar@{--}[r]\ar@[green][d]\ar@[red][ul]  & {} \stackrel{(s-2,1)}{\circ} \ar[r]\ar@[green][d]  & \stackrel{(s,1)}{\circ} \ar@/^/[lllll]\ar@[green][d] \ar@[red][ul] \\
\stackrel{(0,2)}{\circ} \ar[r]\ar@[green][d]\ar@[red][urrrrr] & \stackrel{(1,2)}{\circ} \ar[r]\ar@[green][d]\ar@[red][ul]  & \stackrel{(2,2)}{\circ} \ar[r]\ar@[green][d]\ar@[red][ul]  & \stackrel{(3,2)}{\circ} \ar@{--}[r]\ar@[green][d]\ar@[red][ul]  & {} \stackrel{(s-2,2)}{\circ} \ar[r]\ar@[green][d]  & \stackrel{(s,2)}{\circ} \ar@/^/[lllll]\ar@[green][d] \ar@[red][ul] \\
\stackrel{(0,3)}{\circ} \ar[r]\ar@{--}[d]\ar@[red][urrrrr] & \stackrel{(1,3)}{\circ} \ar[r]\ar@{--}[d]\ar@[red][ul]  & \stackrel{(2,3)}{\circ} \ar[r]\ar@{--}[d]\ar@[red][ul]  & {} \stackrel{(3,3)}{\circ} \ar@{--}[r]\ar@{--}[d]\ar@[red][ul]  & \stackrel{(s-2,3)}{\circ} \ar[r]\ar@{--}[d]  & \stackrel{(s,3)}{\circ} \ar@/^/[lllll]\ar@{--}[d] \ar@[red][ul] \\ %{}
\stackrel{(0,r-2)}{\circ} \ar[r]\ar@[green][d] & \stackrel{(1,r-2)}{\circ} \ar[r]\ar@[green][d] & \stackrel{(2,r-2)}{\circ} \ar[r]\ar@[green][d] & \stackrel{(3,r-2)}{\circ} \ar@{--}[r]\ar@[green][d] & {} \stackrel{(s-2,r-2)}{\circ} \ar[r]\ar@[green][d]  & \stackrel{(s,r-2)}{\circ} \ar@/^/[lllll]\ar@[green][d]  \\
\stackrel{(0,r-1)}{\circ} \ar[r]\ar@[green]@/^/[uuuuu]\ar@[red][urrrrr] &
\stackrel{(1,r-1)}{\circ} \ar[r]\ar@[green]@/^/[uuuuu] \ar@[red][ul] &
\stackrel{(2,r-1)}{\circ} \ar[r]\ar@[green]@/^/[uuuuu] \ar@[red][ul]&
\stackrel{(3,r-1)}{\circ} \ar@{--}[r]\ar@[green]@/^/[uuuuu] \ar@[red][ul]& {} \stackrel{(s-2,r-1)}{\circ} \ar[r]\ar@[green]@/^/[uuuuu]  & \stackrel{(s-1,r-1)}{\circ} \ar@/^/[lllll]\ar@[green]@/^/[uuuuu]\ar@[red][ul]  \\
}
}

For each vertex $\bi \in \tQ_0$, write
$$z(\xc,\bi,c_{\bi}) =c_{\bi}\xb_{\bi+e_1}\xa_{\bi} + \xa_{\bi+xe_2}\xb_{\bi},$$
\eqqcn{relGrsb}{z(\xb,\bi,b_\bi)=b_\bi\xa_{\bi-\be}\xc_\bi + \xc_{\bi+\be_1}\xa_\bi,} and \eqqcn{relGrsa}{z(\xa,\bi,a_\bi)=a_\bi\xb_{\bi-\be}\xc_\bi + \xc_{\bi+\be_2}\xb_\bi, } for $a_\bi, b_\bi, c_{\bi}\in k^*$.
Let $C(a,b,c) = \{a_\bi,b_\bi,c_\bi| \bi\in \tQ_0\}$ be a subset of $k^*$ and let
\eqqcn{rhoQsrcz}{\arr{lll}{ \trho_{comm}(s,r, C(a,b,c)) & = & \{z(\xa,\bi,a_\bi), z(\xb,\bi,b_\bi),z(\xc,\bi,c_\bi)|\bi\in \tQ_0\}\\
 \trho_{zero}(s,r) & = & \{\xa_{\bi+\be_1}\xa_\bi, \xb_{\bi+\be_2}\xb_\bi, \xc_{\bi-\be}\xc_\bi|\bi\in \tQ_0\},}}
and let \eqqc{rhoQsr}{\trho(s,r, C(a,b,c)) = \trho_{comm}(s,r,C(a,b,c))\cup \trho_{zero}(s,r).}

\begin{pro}\label{relQsr2tr}
If the quotient algebra $\tL(s,r) = k\tQ(s,r)/I$ is a $2$-translation  algebra, then there is \eqqcnn{}{\{a_\bi,b_\bi,c_\bi| \bi\in \tQ_0\}\subseteq k^*,} such that $I$ is generated by the set $\trho(s,r,C(a,b,c))$ in \eqref{rhoQsr}.
\end{pro}
\begin{proof}
Assume that $\tL(s,r) = k\tQ/I$ is a $2$-translation algebra.

Consider the square with vertices $\bi, \bi+\be_1, \bi+\be_2, \bi+\be $. Since there is an arrow $\xc_{\bi}$ from $\bi$ to $\bi-\be$, one gets that there are $c_\bi,c'_\bi$ such that \eqqcn{relGrsc}{ c_\bi\xb_{\bi+\be_1}\xa_\bi + c'_\bi\xa_{\bi+\be_2}\xb_\bi \in I,} by  (2) of Proposition \ref{commrel}.
Since $\tQ(G)$ is $2$-translation quiver with trivial $2$-translation, and $e_\bi \tL_1 (G) e_{\bi+\be} \times e_{\bi+\be}\tL_2(G) e_\bi \to e_\bi \tL_3 e_\bi$ defines an non degenerated bilinear form, we have that $c_\bi \neq 0\neq c'_\bi$.
And we may take $c'_\bi=1$, so we get  $$z(\xc,\bi,c_{\bi}) = z(\xc,\bi,c_{\bi},1)\in I.$$

Similarly, there are $a_\bi ,b_\bi \in k^*$ such that  \eqqcn{relGrsb}{z(\xb,\bi,b_\bi)=b_\bi\xa_{\bi-\be}\xc_\bi + \xc_{\bi+\be_1}\xa_\bi \in I,} and \eqqcn{relGrsa}{z(\xa,\bi,a_\bi)=a_\bi\xb_{\bi-\be}\xc_\bi + \xc_{\bi+\be_2}\xb_\bi \in I.}

For each $\bi\in \tQ_0$, since $\bi \neq \bi+2\be_t$ for $1\le t \le 3$, we have that \eqqcn{}{\xa_{\bi+\be_1}\xa_\bi, \xb_{\bi+\be_2}\xb_\bi, \xc_{\bi-\be}\xc_\bi \in I,} by Proposition \ref{zerorel}.

So $\trho(s,r,C(a,b,c)) \subset I$.

Let $\tL=k\tQ/(\trho(s,r,C(a,b,c)))$, then $\tL(s,r)$ is a factor algebra of $\LL$.
Then we have that $\dmn_k e_j\tL_1 e_i = \dmn_k e_j\tL_1(G(s,r)) e_i $ for all $i,j$ since $\tL$ and $\tL(s,r)$ have the same quiver.
In fact, by using the same notation for a path in $\tQ$ and its image in $\tL$, for each $\bi \in \tQ_0 = \zZ/s\zZ \times \zZ/r\zZ$, we have that
$$ e_\bi \tL_2 = k \xb_{\bi-\be_2}\xa_{\bi-\be} + k \xc_{\bi+\be}\xb_{\bi+be_1} + k \xc_{\bi+\be}\xa_{\bi+be_2}, $$

$$\tL_3 e_\bi = k \xc_{\bi+\be}\xb_{\bi+\be_1}\xa_\bi,$$
by computing directly using the relations in $\trho(s,r,C(a,b,c))$.
So we have $$\dmn_k e_\bi \tL_3 e_\bi \le 1 = \dmn_k e_\bi \tL_3(G(s,r)) e_{\bi}, $$
and $$\dmn_k e_{\bi} \tL_2 e_{\bi'}  \left\{ \arr{ll}{\le 1, & \mbox{for }\bi' = \bi-\be, \bi+\be_1,\bi+\be_2  \\ = 0, & \mbox{otherwise} }\right.$$
This implies that $$\dmn_k e_{\bi} \tL_2 e_{\bi'} \le \dmn_k e_{\bi} \tL_2 e_{\bi'},$$ for any $  \bi, \bi'$, by comparing the Loewy matrix of $\tL(s,r)$.

So $$e_{\bi'} \tL_t e_{\bi} =e_{\bi'} \tL_t(G(s,r)) e_{\bi}$$ for all $\bi,\bi'$ and $ \tL= \tL(s,r)$.

This proves $I=(\trho(s,r,C(a,b,c)))$ is generated by $\trho(s,r,C(a,b,c))$.
\end{proof}

For the quadratic dual quiver $\tQ^{\perp} (s,r)$, we have the following  Proposition describe its relations, by an easy calculation.
\begin{pro}\label{relQsrAS} If
the quotient algebra $\tG(s,r)= k\tQ(s,r)/I'$ is Koszul AS-regular algebra of finite Gelfand-Kirilov dimension, then the relation set can be chosen as
\eqqc{rhoQsrperp}{ \trho^{\perp}(s,r,C(a,b,c)) = \{z(\xa,\bi,-a^{-1}_\bi), z(\xb,\bi,-b^{-1}_\bi),z(\xc,\bi,-c^{-1}_\bi)|\bi\in \tQ_0\}.}
\end{pro}
\begin{proof}
Let $\tL(s,r)= k\tQ(s,r)/(\trho(s,r))$ for some $\trho(s,r,C(a,b,c))$ as defined in \eqref{rhoQsr}.
It is immediate that $\trho(s,r,C(a,b,c))$ and $\trho^{\perp}(s,r,C(a,b,c))$ are both linear independent sets of elements in the subspace $k \tQ_2(s,r)$ spanned by paths of length $2$, and by identifying the space $k\tQ_2(s,r)$ of the paths of length $2$ in $\tQ(s,r)$ with its dual space, i.e., $x(y)=0$ for $x\in \trho^{\perp}(s,r,C(a,b,c))$ and $y\in \trho(s,r,C(a,b,c))$.
For each pair $\bi,\bj\in \tQ_0(s,r)$, it is immediate that $\dmn_k k \tQ_2(s,r) = |\trho^{\perp}_{\bi,\bj}(s,r,C(a,b,c))| + |\trho_{\bi,\bj}(s,r,C(a,b,c))|,$ so $\trho(s,r,C(a,b,c))$ and $\trho^{\perp}(s,r,C(a,b,c))$ span orthogonal subspaces in $k \tQ_2(s,r)$.
By Proposition \ref{relQsr2tr}, we have that if $\tL(s,r)= k\tQ(s,r)/I$ is a Koszul $2$-translation algebra then $I$ is generated by $\trho(s,r,C(a,b,c))$.
Thus if $\tG(s,r)= k\tQ(s,r)/I'$ is Koszul AS-regular algebra of finite Gelfand-Kirilov dimension then $I'$ is generated by $\trho^{\perp}(s,r,C(a,b,c))$ for some $C(a,b,c)\subset k^*$, by Theorem 5.1 of \cite{m99}.
\end{proof}

Now construct the quiver $\zZ_{\vv} \tQ(s,r)$ from $\tQ(s,r)$, this is an infinite stable acyclic $2$-translation quiver.
It has $3$ connected components when $s,r$ are both multiples of $3$, and has only one connected component otherwise.

Recall by \eqref{relationzv}, the relation set for $\zZ_{\vv} \tQ(s,r)$ is defined as  $\rho_{\zZ_{\vv} \tQ} =\{\zeta[m]| \zeta \in \trho, m \in \mathbb Z \}$.
Write
$$z(\xc,(\bi,t),c_{\bi}) =c_{\bi}\xb_{\bi+\be_1,t} + \xa_{\bi+\be_2,t+1}\xb_{\bi,t},$$
\eqqcn{relGrsb}{z(\xb,(\bi,t),b_\bi)=b_{\bi}\xa_{\bi-\be,t+1}\xc_{\bi,t} + \xc_{\bi+\be_1,t+1}\xa_{\bi,t},} and \eqqcn{relGrsa}{z(\xa,(\bi,t),a_\bi)=a_\bi\xb_{\bi-\be,t+1}\xc_{\bi,t} + \xc_{\bi+\be_2,t+1}\xb_{\bi,t}, } for $a_\bi, b_\bi, c_\bi \in k^*$.
So we have for each $t\in \zZ$,
\eqqcn{rhoQsrczt}{\arr{lll}{ \trho_{comm}(s,r,C(a,b,c))[t] & = & \{z(\xa,(\bi,t),a_\bi), z(\xb,(\bi,t),b_\bi),z(\xc,(\bi,t),c_\bi)|\bi\in \tQ_0\}\\
 \trho_{zero}(s,r)[t] & = & \{\xa_{\bi+\be_1,t+1}\xa_{\bi,t}, \xb_{\bi+\be_2,t+1}\xb_{\bi,t}, \xc_{\bi-\be,t+1}\xc_{\bi,t}|\bi\in \tQ_0\},}}
and let \eqqc{rhoQsrt}{\rho_{\zZ_{\vv} \tQ(s,r)} = \bigcup_{t\in \zZ}(\trho_{comm}(s,r,C(a,b,c))[t]\cup \trho_{zero}(s,r)[t]).}

By taking a connected component $Q(s,r)$ in the complete $\tau$-slice $\zZ_{\vv} \tQ(s,r)[0,2]$ of $\zZ_{\vv} \tQ(s,r)$, one obtains a $2$-nicely-graded quiver $Q(s,r)$ as follows.

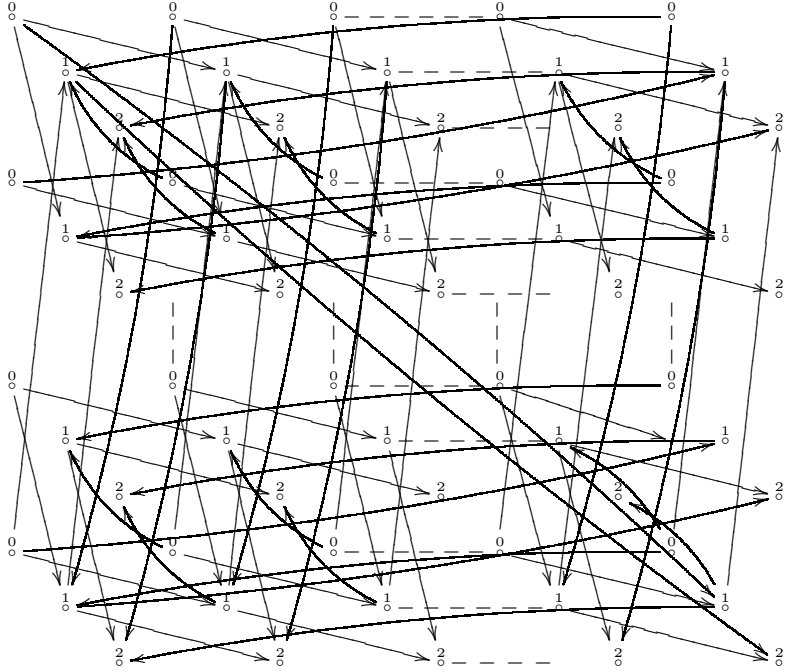
\footnotesize\Tiny
\eqqcn{McKayquiverGsr}{
\xymatrix@C=0.4cm@R0.3cm{
%&&&&&&&Q(s,r)\\
\stackrel{0}{\circ} \ar[rrrrd] \ar@[green][ddddr] \ar@[purple]@/^/[dddddddddddrrrrrrrrrrrrr] & & & \stackrel{0}{\circ} \ar[rrrrd] \ar@[green][ddddr] \ar@[purple]@/^/[dddddddddddll] & & & \stackrel{0}{\circ} \ar@[white]@{--}[rrr] \ar@[green][ddddr] \ar@[purple]@/^/[dddddddddddll] & & & {} \stackrel{0}{\circ} \ar[rrrrd] \ar@[green][ddddr] & & & \stackrel{0}{\circ} \ar@/_/[dlllllllllll] \ar@[purple]@/^/[dddddddddddll] \\%1a
&\stackrel{1}{\circ}  \ar[rrrrd] \ar@[green][ddddr] \ar@[purple]@/_/[dddddddddddrrrrrrrrrrrrr] & & & \stackrel{1}{\circ} \ar[rrrrd] \ar@[green][ddddr] \ar@[purple]@/^/[dddddddddddll] & & & \stackrel{1}{\circ} \ar@{--}[rrr] \ar@[green][ddddr] \ar@[purple]@/^/[dddddddddddll] & & & {} \stackrel{1}{\circ} \ar[rrrrd]\ar@[green][ddddr] & & & \stackrel{1}{\circ} \ar@/_/[dlllllllllll]\ar@[purple]@/^/[dddddddddddll] \\%1b
& & \stackrel{2}{\circ}  & & & \stackrel{2}{\circ}  & & & \stackrel{2}{\circ} \ar@[white]@{--}[rr] & & & {} \stackrel{2}{\circ}  & & & \stackrel{2}{\circ} \\%1c
\stackrel{0}{\circ} \ar[rrrrd] \ar@[purple]@/_/[uurrrrrrrrrrrrr] & & & \stackrel{0}{\circ} \ar@[purple]@/^/[uull] \ar[rrrrd] & & & \stackrel{0}{\circ} \ar@[purple]@/^/[uull]\ar@[white]@{--}[rrr] & & & {} \stackrel{0}{\circ} \ar[rrrrd] & & & \stackrel{0}{\circ}\ar@[purple]@/^/[uull] \ar@/_/[dlllllllllll]\\%2a
&\stackrel{1}{\circ} \ar[rrrrd] \ar@[purple]@/_/[uurrrrrrrrrrrrr] & & & \stackrel{1}{\circ} \ar[rrrrd]\ar@[purple]@/^/[uull] & & & \stackrel{1}{\circ} \ar@[purple]@/^/[uull]\ar@{--}[rrr] & & & {} \stackrel{1}{\circ} \ar[rrrrd] & & & \stackrel{1}{\circ} \ar@[purple]@/^/[uull] \ar@/_/[dlllllllllll] \\%2b
& & \stackrel{2}{\circ}  &\ar@{--}[dd] & & \stackrel{2}{\circ}  &\ar@{--}[dd] & & \stackrel{2}{\circ} \ar@[white]@{--}[rr] &\ar@{--}[dd] & & {} \stackrel{2}{\circ}  &\ar@{--}[dd] & & \stackrel{2}{\circ} \\%2c
&&&&&&&&&&&&&&&&\\
\stackrel{0}{\circ} \ar[rrrrd] \ar@[green][ddddr] & & & \stackrel{0}{\circ} \ar[rrrrd] \ar@[green][ddddr]  & & & \stackrel{0}{\circ} \ar@[white]@{--}[rrr] \ar@[green][ddddr] & & & {} \stackrel{0}{\circ} \ar[rrrd]\ar@[green][ddddr] & & & \stackrel{0}{\circ} \ar@/_/[dlllllllllll] \\%r-2a
&\stackrel{1}{\circ} \ar[rrrrd] \ar@[green][ddddr] & & & \stackrel{1}{\circ} \ar[rrrrd] \ar@[green][ddddr] & & & \stackrel{1}{\circ} \ar@{--}[rrr]\ar@[green][ddddr] & & & {} \stackrel{1}{\circ} \ar[rrrrd]\ar@[green][ddddr] & & & \stackrel{1}{\circ} \ar@/_/[dlllllllllll] \\%r-2b
& & \stackrel{2}{\circ}  & & & \stackrel{2}{\circ}  & & & \stackrel{2}{\circ}  & & & {} \stackrel{2}{\circ}  & & & \stackrel{2}{\circ} \\%r-2c
\stackrel{0}{\circ} \ar[rrrrd]\ar@[purple]@/_/[uurrrrrrrrrrrrr] \ar@[green][ruuuuuuuuu]& & & \stackrel{0}{\circ} \ar[rrrrd] \ar@[purple]@/^/[uull] \ar@[green][ruuuuuuuuu] & & & \stackrel{0}{\circ} \ar@[purple]@/^/[uull]\ar@[white]@{--}[rrr] \ar@[green][ruuuuuuuuu] & & & {} \stackrel{0}{\circ} \ar[rrrrd] \ar@[green][ruuuuuuuuu] & & & \stackrel{0}{\circ}\ar@[purple]@/_/[uull] \ar@/_/[dlllllllllll] \ar@[green][ruuuuuuuuu] \\%r-1a
&\stackrel{1}{\circ} \ar[rrrrd]\ar@[purple]@/_/[uurrrrrrrrrrrrr] \ar@[green][ruuuuuuuuu] & & & \stackrel{1}{\circ}\ar@[purple]@/^/[uull] \ar[rrrrd]\ar@[green][ruuuuuuuuu] & & & \stackrel{1}{\circ} \ar@[purple]@/^/[uull] \ar@[green][ruuuuuuuuu]\ar@{--}[rrr] & & & {} \stackrel{1}{\circ} \ar[rrrrd] \ar@[green][ruuuuuuuuu]& & & \stackrel{1}{\circ} \ar@[purple]@/_/[uull] \ar@/_/[dlllllllllll] \ar@[green][ruuuuuuuuu] \\%r-1b
& & \stackrel{2}{\circ}  & & & \stackrel{2}{\circ}  & & & \stackrel{2}{\circ} \ar@[white]@{--}[rr] & & & {} \stackrel{2}{\circ}  & & & \stackrel{2}{\circ} \\%r-1c
}
}
%}
\normalsize
Here we denote by $\stackrel{d}{\circ}$  the vertex $(\bi,d)$ in $\zZ_\vv \tQ$, so $\xa_{\bi,d}$ is the arrow from $(\bi,d) $ to $(\bi+\be_1,d+1)$, $\xb_{\bi,d}$ is the arrow from $(\bi,d) $ to $(\bi+\be_2,d+1)$ and $\xc_{\bi,d}$ is the arrow from $(\bi,d) $ to $(\bi-\be,d+1)$.

Let $e_{[0,2]} =\sum_{\bi \in\zZ_{\vv} \tQ_0(s,r)[0,2]} e_{\bi}$, the relation set $$\rho(s,r)= \{x\in \rho_{\zZ_{\vv} \tQ(s,r)}| e_{[0,2]} x e_{[0,2]} = x \} = \trho_{comm}(s,r,C(a,b,c))[0]\cup \trho_{zero}(s,r)[0].$$
Since any sequence of relations in $\trho_{comm}(s,r)[0]$ such that the successive pair share an arrow does not form a circle, by changing the representative for the arrows, the coefficient in the commutative relations in $\rho(s,r)$ can all be taken as $-1$, and
\eqqc{tsrcommrel}{\arr{lll}{\rho(s,r)& = &
\{\xa_{\bi+\be_1,1}\xa_{\bi,0},\xb_{\bi+\be_2,1}\xb_{\bi,0}, \xc_{\bi-\be_1,1}\xc_{\bi,0}|\bi \in \zZ/s\zZ\times \zZ/r\zZ\} \\ & &\cup \{\xa_{\bi+\be_2,1}\xb_{\bi,0} - \xb_{\bi+\be_1,1}\xa_{\bi,0}, \xa_{\bi-\be,1}\xc_{\bi,0} - \xc_{\bi+\be_1,1}\xa_{\bi,0},\\ & & \quad \xa_{\bi-\be_1,1}\xc_{\bi,0} - \xc_{\bi+\be_1,1}\xa_{\bi,0} | \bi \in \zZ/s\zZ\times \zZ/r\zZ\}.}}

\begin{pro}\label{2slcrel} If the quotient algebra $\LL(s,t)= kQ(s,t)/ I$ is a $2$-properly-graded algebra, then the ideal $I$ is generated by $\rho(s,r)$.
\end{pro}

The quadratic dual quiver $Q^{\perp}(s,r)=(Q_0(s,r), Q_1(s,r), \rho^{\perp}(s,r))$ is a $2$-slice quiver associated to the McKay quiver $\tQ(s,r)$.
It is easy to check that
\eqqc{srcommrel}{\arr{lll}{\rho^{\perp}(s,r) & = & \{\xa_{\bi+\be_2,d+1}\xb_{\bi,d_1} = \xb_{\bi+\be_1,d+1}\xa_{\bi,d}, \xa_{\bi-\be,1}\xc_{\bi,0} + \xc_{\bi+\be_1,1}\xa_{\bi,0},\\ & & \qquad \xa_{\bi-\be_1,1}\xc_{\bi,0} + \xc_{\bi+\be_1,1}\xa_{\bi,0} | \bi \in \zZ/s\zZ\times \zZ/r\zZ \}.}}
\begin{pro}\label{2slcrelp} If the quotient algebra $\GG(s,t)= kQ(s,t)/ I'$ is a $2$-slice algebra then $I'$ is generated by $\rho^{\perp}(s,r)$.
\end{pro}
So the relations for the $n$-slice algebra are independent of the parameter set.

\subsection{Relations for $\tQ(\Xi)$}
Let $G(\Xi)$ be a subgroup of $\mathrm{SL}(2,\mathbb C)$ with McKay quiver $\Xi$.
Let $\tQ(\Xi)$ be the McKay quiver obtained by embedding $G(\Xi)$ in $\mathrm{SL}(3,\mathbb C)$ as in Proposition \ref{embedmcq}.
Then $\tQ(\Xi)$ is one of the following quivers.

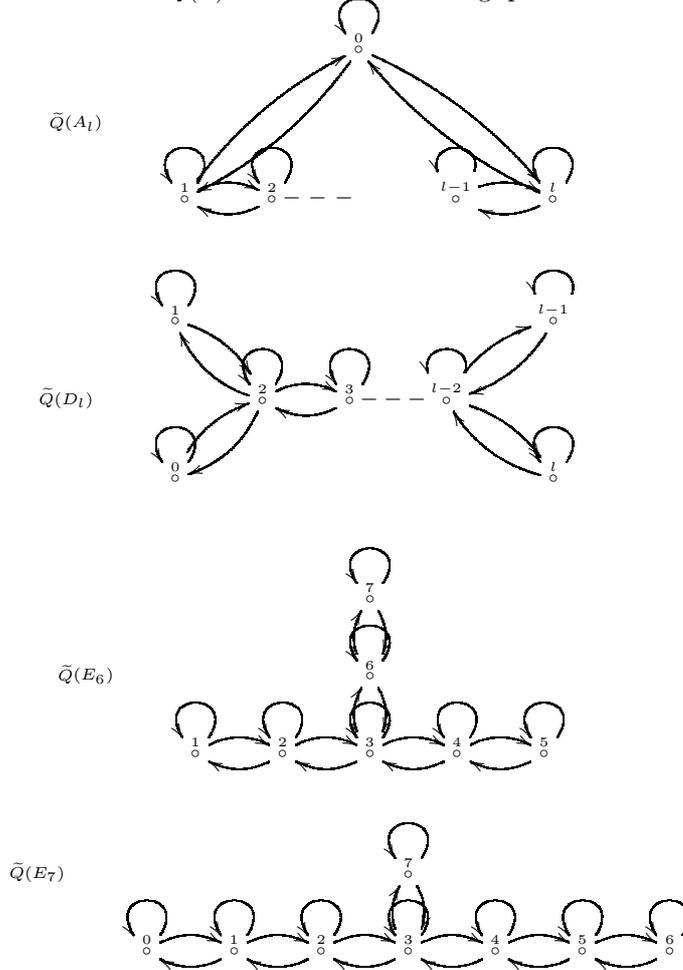
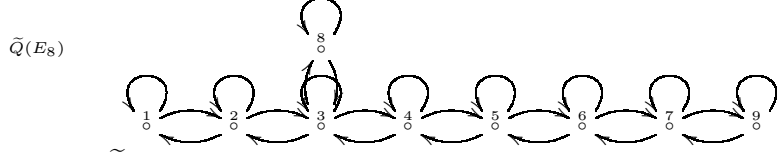
\tiny$$
\xymatrix@C=0.8cm@R0.5cm{
&&&\stackrel{0}{\circ}\ar@[purple]@(ur,ul)[]\ar@[green]@/^/[ddll]\ar@/^/[ddrr]&&&\\
\tQ(A_l) &&&&&&\\
& \stackrel{1}{\circ}\ar@[purple]@(ur,ul)[] \ar@[green]@/^/[r]\ar@/^/[uurr] & \stackrel{2}{\circ}\ar@[purple]@(ur,ul)[]\ar@/^/[l]\ar@{--}[r] & &\stackrel{l-1}{\circ}\ar@[purple]@(ur,ul)[] \ar@[green]@/^/[r] &\stackrel{l}{\circ}\ar@[purple]@(ur,ul)[] \ar@/^/[l]\ar@[green]@/^/[uull]\\
&&&&&&&\\
}$$\tiny$$\xymatrix@C=0.8cm@R0.5cm{
&\stackrel{1}{\circ}\ar@[purple]@(ur,ul)[]\ar@[green]@/^/[dr]&&&&\stackrel{l-1}{\circ}\ar@[purple]@(ur,ul)[] \ar@/^/[dl]&&\\
\tQ(D_l) &&\stackrel{2}{\circ}\ar@[purple]@(ur,ul)[]\ar@/^/[ul]\ar@/^/[dl] \ar@[green]@/^/[r] & \stackrel{3}{\circ}\ar@[purple]@(ur,ul)[]\ar@/^/[l] \ar@{--}[r]&\stackrel{l-2}{\circ}\ar@[purple]@(ur,ul)[]\ar@[green]@/^/[ur] \ar@[green]@/^/[dr] &\\
&  \stackrel{0}{\circ}\ar@[purple]@(ur,ul)[]\ar@[green]@/^/[ur] & & & &\stackrel{l}{\circ}\ar@[purple]@(ur,ul)[] \ar@/^/[ul]&&\\
&&&&&&&\\
}$$
\tiny$$\xymatrix@C=0.8cm@R0.5cm{
&&&\stackrel{7}{\circ}\ar@[purple]@(ur,ul)[]\ar@/^/[d]&&&\\
\tQ(E_6) &&&\stackrel{6}{\circ}\ar@[purple]@(ur,ul)[]\ar@/^/[d]\ar@[green]@/^/[u]&&&&\\
& \stackrel{1}{\circ}\ar@[purple]@(ur,ul)[] \ar@[green]@/^/[r] & \stackrel{2}{\circ}\ar@[purple]@(ur,ul)[]\ar@/^/[l]\ar@[green]@/^/[r] & \stackrel{3}{\circ}\ar@[purple]@(ur,ul)[]\ar@[green]@/^/[u] \ar@[green]@/^/[r] \ar@/^/[l] &\stackrel{4}{\circ}\ar@[purple]@(ur,ul)[] \ar@[green]@/^/[r] \ar@/^/[l] &\stackrel{5}{\circ}\ar@[purple]@(ur,ul)[] \ar@/^/[l]\\
&&&&&&&\\}$$
\tiny$$ \xymatrix@C=0.8cm@R0.5cm{
\tQ(E_7) &&&&\stackrel{7}{\circ}\ar@[purple]@(ur,ul)[]\ar@/^/[d]&&&&\\
& \stackrel{0}{\circ}\ar@[purple]@(ur,ul)[] \ar@[green]@/^/[r] & \stackrel{1}{\circ}\ar@[purple]@(ur,ul)[]\ar@/^/[l]\ar@[green]@/^/[r] & \stackrel{2}{\circ}\ar@[purple]@(ur,ul)[] \ar@[green]@/^/[r] \ar@/^/[l] &\stackrel{3}{\circ}\ar@[purple]@(ur,ul)[]\ar@[green]@/^/[u] \ar@[green]@/^/[r] \ar@/^/[l]&\stackrel{4}{\circ}\ar@[purple]@(ur,ul)[] \ar@[green]@/^/[r] \ar@/^/[l] &\stackrel{5}{\circ}\ar@[purple]@(ur,ul)[] \ar@[green]@/^/[r] \ar@/^/[l]&\stackrel{6}{\circ}\ar@[purple]@(ur,ul)[] \ar@/^/[l]\\
&&&&&&&\\
}$$
\tiny$$ \xymatrix@C=0.8cm@R0.5cm{
\tQ(E_8) &&&\stackrel{8}{\circ}\ar@[purple]@(ur,ul)[]\ar@/^/[d]&&&&\\
& \stackrel{1}{\circ}\ar@[purple]@(ur,ul)[] \ar@[green]@/^/[r] & \stackrel{2}{\circ}\ar@[purple]@(ur,ul)[]\ar@/^/[l]\ar@[green]@/^/[r] & \stackrel{3}{\circ}\ar@[purple]@(ur,ul)[]\ar@[green]@/^/[u] \ar@[green]@/^/[r] \ar@/^/[l] &\stackrel{4}{\circ}\ar@[purple]@(ur,ul)[] \ar@[green]@/^/[r] \ar@/^/[l] &\stackrel{5}{\circ}\ar@[purple]@(ur,ul)[] \ar@[green]@/^/[r] \ar@/^/[l]&\stackrel{6}{\circ}\ar@[purple]@(ur,ul)[] \ar@[green]@/^/[r] \ar@/^/[l] &\stackrel{7}{\circ}\ar@[purple]@(ur,ul)[] \ar@[green]@/^/[r] \ar@/^/[l]&\stackrel{9}{\circ}\ar@[purple]@(ur,ul)[] \ar@/^/[l]
}
$$\normalsize
The vertices for $\tQ(A_l)$ are indexed by $\zZ/(l+1)\zZ$.
Write $\xc_i$ for the loop at the vertex $i$, $\xa_{i,j}$ for the arrow from $i$ to $j$ with $j>i$ ($j=i+1$ in the case $A_n$) and $\xb_{i,j}$ for the arrow from $i$ to $j$ with $j<i$ ($j=i-1$ in the case $A_n$).
For the sake of writing down the arrows properly, we do not follow the convention to index the vertices.

We have immediate the following on the arrows of these McKay quivers.

\begin{lem}\label{McKay_arrows}
Let $\tQ(\Xi)$ be the McKay quiver defined above for $\Xi = A_l, D_l, E_6, E_7, E_8$, $l\ge 5$.
\begin{enumerate}
\item There is a loop at each vertex of $\tQ(\Xi)$.

\item There is at most one arrow from $i$ to $j$ for any two vertices $i,j$ in $\tQ(\Xi)$.

\item There is an arrow $\xa_{i,j}$ from $i$ to $j$ if and only if there is an arrow $\xb_{j,i}$ from $j$ to $i$.

\item There are at most $3$ arrows leaving a vertex, and at most $3$ arrows heading to a vertex.
\end{enumerate}
\end{lem}

Let $\tL(\Xi) \simeq k\tQ(\Xi)/I(\Xi)$ be a $2$-translation algebra for the ideal $I(\Xi)$, we determine relations $\trho(\Xi)$ for $\Xi \neq A_2, A_3,A_4,D_4$.

\begin{lem}\label{diffih}
If $i\neq h$, then
\eqqcn{}{\xa_{j,h}\xa_{i,j}, \xb_{j,h}\xb_{i,j}, \xa_{j,h}\xb_{i,j}, \xb_{j,h}\xa_{i,j}\in I(\Xi)} if such paths exist.
\end{lem}
\begin{proof}
If $\xa_{i,j},\xa_{j,h} \in \tQ_1(\Xi)$, one sees from the quivers that $\xa_{j,h}\xa_{i,j}$ is the only path of length $2$ from $i$ to $h$, and there is no arrow from $h$ to $i$.
By using Proposition \ref{zerorel}, we have that \eqqcn{}{\xa_{j,h}\xa_{i,j} \in I(\Xi).}

Similarly we have
\eqqcn{}{\xb_{j,h}\xb_{i,j}, \xa_{j,h}\xb_{i,j}, \xb_{j,h}\xa_{i,j}\in I(\Xi)} if such paths exist.
\end{proof}

Write \eqqcn{}{\arr{lll}{ \trho_{11}(\Xi) &=& \{\xa_{j,h} \xa_{i,j}|| i,h ,j \in \tQ_0(\Xi), i\neq h \},\\
\trho_{12}(\Xi) &=& \{ \xa_{j,h}\xb_{i,j}| i,h ,j \in \tQ_0(\Xi), i\neq h  \} \\
\trho_{21}(\Xi) &=& \{ \xa_{j,h} \xb_{i,j}| i,h ,j \in \tQ_0(\Xi), i\neq h\}\\
\trho_{22}(\Xi) &=& \{ \xb_{j,h}\xb_{i,j} | i,h ,j \in \tQ_0(\Xi), i\neq h  \}\\
}
}
Take \eqqc{twoarrows}{ \trho_p(\Xi) = \trho_{11}(\Xi) \cup \trho_{12}(\Xi) \cup\trho_{21}(\Xi) \cup \trho_{22}(\Xi)    .}
As a corollary of Lemma \ref{diffih}, we get the following.
\begin{pro}\label{rho_diffih}
$\trho_p(\Xi) \subseteq I(\Xi)$.
\end{pro}

Let $C_a = \{a_{i,j}\in k^*| \xa_{i,j}\in\tQ_{1}(\Xi)\}$, that is, choose a nonzero number $a_{i,j}$ for each arrow $\xa_{i,j}$, and set \eqqc{onearrow}{\trho_a(\Xi, C_a) =\{\xa_{i,j}\xc_i - \xc_j\xa_{i,j},  \xb_{j,i}\xc_j -  \xc_i\xb_{j,i} | \xa_{i,j}\in \tQ_1, a_{i,j}\in C_a,i<j \}.}

\begin{pro}\label{rho_onearrow}
By choosing the representatives of the arrows suitably, we have a set $ C_a\subset k^*$ such that
$$\trho_a(\Xi,C_a)  \subseteq I(\Xi).$$
\end{pro}
\begin{proof}
By Lemma \ref{McKay_arrows},  there is an arrow $\xa_{i,j}$ from $i$ to $j$ if and only if there is an arrow $\xb_{j,i}$ from $j$ to $i$ in $\tQ(\Xi)$.
By Proposition \ref{commrel}, we see that for each arrow $\xa_{i,j}$, we have that \eqqcnn{}{\xb_{j,i}\xc_j, \xc_i\xb_{j,i}} are bound paths and \eqqcn{}{\xb_{j,i}\xc_j - b_{j,i} \xc_i\xb_{j,i} \in I(\Xi),} for some $b_{j,i}\neq 0$, similar to the proof of Proposition \ref{relQsr2tr}.
Now apply Proposition \ref{commrel} for the arrow  $\xb_{j,i}$, we have  that \eqqcnn{}{\xa_{i,j}\xc_j, \xc_j\xa_{i,j}} are bound paths and \eqqcn{}{\xa_{i,j}\xc_i - a'_{i,j} \xc_j\xa_{i,j} \in I(\Xi),} for some $a'_{i,j}\neq 0$.
Take $C_a =\{a'_{i,j}|\xb_{j,i}\in \tQ_1(\Xi)\}$.

Starting from $i=0$ in cases $\Xi= A_l,D_l,E_7$ and from $i=0$ in cases $\Xi= E_6,E_8$, by changing the representatives of arrows $\xc_j$ by $b_{j,i}\xc_j $ one by one,  we get  $b_{j,i} = 1$ for all $\xa_{i,j}$.
That is, $$\xb_{j,i}\xc_j -  \xc_i\xb_{j,i} \in I(\Xi)$$  for  all arrows $\xa_{i,j}$ in $\tQ(\Xi)$.

This proves that by choosing the representatives of the arrows suitably, we have  $\trho_a(\Xi, C_a)  \subseteq I(\Xi).$
\end{proof}

Write $\tQ^{i,j}_t(\Xi)$ for the set of paths of length $t$ from $i$ to $j$ in $\tQ(\Xi)$.
The set $\tQ^{i,j}_t(\Xi)$ is  a orthogonal basis of the subspace $k \tQ^{i,j}_t(\Xi)$ spanned by the paths of length $t$ in $k\tQ(\Xi)$.
We have immediately the following.
\begin{lem}\label{orthogonal_in_ij}
For each arrow $\xa_{i,j}$ and nonzero $a_{i,j}\in k^*$,  $a_{i,j}\xa_{i,j}\xc_i +\xc_j\xa_{i,j},  \xb_{j,i}\xc_j + \xc_i\xb_{j,i}$ is a basis of the orthogonal subspace of $\{\xa_{i,j}\xc_i - a_{i,j} \xc_j \xa_{i,j},  \xb_{j,i}\xc_j - \xc_i\xb_{j,i}\}$ in the space $ k\tQ^{i,j}_2(\Xi) $ spanned by paths of length $2$ of $k\tQ(\Xi)$.
\end{lem}

Write \small\eqqc{onearrowp}{{\trho_a}^\perp(\Xi,C_a) =\{a_{i,j}\xa_{i,j}\xc_i + \xc_j\xa_{i,j}, \xb_{j,i}\xc_j + \xc_i\xb_{j,i} | \xa_{i,j}\in \tQ_1, a_{i,j}\in C_a, i<j \}.}\normalsize

\medskip

Now consider $\trho_{i,i}(\Xi)$.
By Lemma \ref{McKay_arrows}, for each vertex $i$ in $\tQ(\Xi)$, there is a loop at $i$ and the number of arrows heading to $i$ and the number of arrows leaving  $i$ are the same, and there are at most $3$ arrows leaving $i$.

Fix $i\in \tQ_0(\Xi)$, write \eqqc{defmuzeta}{\mu_{i,j}= \left\{\arr{ll}{\xa_{i,j} & i<j,\\ \xb_{i,j} & i>j, }\right.\mbox{ and }\zeta_{j,i}= \left\{\arr{ll}{\xb_{j,i} & i<j,\\ \xa_{j,i} & i>j. }\right.}
Then $\mu_{i,j}$ is an arrow starting at $i$  and $\zeta_{j,i}$ is an arrow heading to $i$.

Consider the following cases.

\begin{lem}\label{loop_one_arrow}
Assume that there is only  one arrow $\mu_{i,j}$ leaves $i$.
\begin{enumerate}
\item If $\xc_i^2 \not\in I(\Xi)$, then there is a $c_i\in k^*$, such that $\xc_i^2 -c_i\zeta_{j,i}\mu_{i,j} \in I(\Xi)$.

\item We have $|\tQ^{i,i}_2(\Xi)|=2$ and for any $c_i\in k^*$ \eqqc{orthbasi}{\{\xc_i^2 -c_i\zeta_{j,i}\mu_{i,j}, c_i\xc_i^2 +\zeta_{j,i} \mu_{i,j} \}} is a orthogonal basis for $kQ^{i,i}_2(\Xi)$.
\end{enumerate}
\end{lem}
\begin{proof}
Apply Proposition \ref{commrel} for the arrow $\xc_i$, the image of  $\xc_i^2$ and $\zeta_{j,i}\mu_{i,j}$ are linearly dependent.
So there is a $c\in k^*$ such that  $\xc_i^2 -c_i\zeta_{j,i}\mu_{i,j}$.

The second assertion follows from  direct computation.
\end{proof}

\begin{lem}\label{loop_two_arrow}
Assume that there are exactly two arrows $\mu_{i,j_1}, \mu_{i,j_2}$ leave $i$.
\begin{enumerate}
\item There is a $ b_i\in k^*$, such that \eqqcnn{}{b_i \zeta_{j_1,i}\mu_{i,j_1 } + \zeta_{j_2,i}\mu_{i,j_2 } \in I(\Xi)  .}

\item If $\xc_i^2 \not \in I(\Xi)$,  then there is a $0\neq c_i \in k$ such that $\xc^2_i + c_i\zeta_{j_1,i}\mu_{i,j_1} \in I(\Xi)$.

\item We have $|\tQ^{i,i}_2(\Xi)|=3$ and for any  $b_i, c_i\in k^*$ \eqqc{orthbasii}{\arr{l}{\{b_i \zeta_{j_1,i}\mu_{i,j_1 } + \zeta_{j_2,i}\mu_{i,j_2 }, \zeta_{j_1,i}\mu_{i,j_1 } - b'_i\zeta_{j_2,i}\mu_{i,j_2 }, \xc^2_i\}\\\mbox{and }\\  \{b_i \zeta_{j_1,i}\mu_{i,j_1 } - \zeta_{j_2,i}\mu_{i,j_2 }, c_i\zeta_{j_1,i}\mu_{i,j_1 } -\xc_i^2, %\\ \quad
     \zeta_{j_1,i}\mu_{i,j_1 } + b_i\zeta_{j_2,i}\mu_{i,j_2 }+ c_i\xc_i^2\}}} are  orthogonal bases for $kQ^{i,i}_2(\Xi)$.
\end{enumerate}
\end{lem}
\begin{proof}
By Proposition \ref{commrel}, we have that the images of  $ \zeta_{j_1,i}\mu_{i,j_1 }$ and of $ \zeta_{j_2,i}\mu_{i,j_2 }$ are bound paths and they are linearly dependent.
So there is a $b_i\in k^*$ such that $\zeta_{j_1,i}\mu_{i,j_1 } - b_i\zeta_{j_2,i}\mu_{i,j_2 }\in I(\Xi)$.

If $\xc_i^2 \not \in I(\Xi)$, then its image is linearly dependent on $\zeta_{j_1,i}\mu_{i,j_1 }$, so  there is a $c_i\in k^*$ such that $\zeta_{j_1,i}\mu_{i,j_1 } - b_i\xc_i^2\in I(\Xi)$.

The rest follow from   direct computations.
\end{proof}

\begin{lem}\label{loop_three_arrow}Assume that  there are three arrows $\mu_{i,j_1}, \mu_{i,j_2}, \mu_{i,j_3}$ leave $i$ with  $j_1< j_2<j_3$.

\begin{enumerate}
\item There are  $ b_i, b_i'\in k^*$, such that \eqqcn{}{b_i \zeta_{j_1,i}\mu_{i,j_1 } + \zeta_{j_2,i}\mu_{i,j_2 }, b'_i \zeta_{j_1,i}\mu_{i,j_1 } + \zeta_{j_3,i}\mu_{i,j_3 }\in I(\Xi)  .}

\item If $\xc_i^2 \not \in I(\Xi)$,  then there is a $c_i \in k^*$ such that $\xc^2_i - c_i\zeta_{j_1,i}\mu_{i,j_1} \in I(\Xi)$.

\item We have $|\tQ^{i,i}_2(\Xi)|=4$ and for any  $b_i,b_i',c_i\in k^*$, \eqqc{orthbasiii}{\arr{l}{
    \{b_i \zeta_{j_1,i}\mu_{i,j_1 } - \zeta_{j_2,i}\mu_{i,j_2 }, b'_i \zeta_{j_1,i}\mu_{i,j_1 } - \zeta_{j_3,i}\mu_{i,j_3}, \\ \quad \xc^2_i - c_i\zeta_{j_1,i}\mu_{i,j_1}, c_i\xc^2_i + \zeta_{j_1,i}\mu_{i,j_1} + b_i \zeta_{j_2,i}\mu_{i,j_2} +b'_i \zeta_{j_3,i}\mu_{i,j_3} \}\\ \mbox{and }\\  \{ \xc^2, b_i \zeta_{j_1,i}\mu_{i,j_1} - \zeta_{j_2,i}\mu_{i,j_2}, b'_i \zeta_{j_1,i}\mu_{i,j_1} - \zeta_{j_3,i}\mu_{i,j_3},\\ \quad \zeta_{j_1,i}\mu_{i,j_1} + b_i \zeta_{j_2,i}\mu_{i,j_2} +b'_i \zeta_{j_3,i}\mu_{i,j_3}, \xc^2 \}
    }
    } are orthogonal bases of $k\tQ_2^{i,i}(\Xi)$.

\end{enumerate}
\end{lem}
\begin{proof}
The lemma follows from Proposition \ref{commrel}, similar to above two lemmas.
\end{proof}

Denote by $\tQ_{01}(\Xi)$ the set of vertices in $\tQ(\Xi)$ from which only one arrow leaves, $\tQ_{02}(\Xi)$ the set of vertices in $\tQ(\Xi)$ from which exact two arrows leave and $\tQ_{03}(\Xi)$ the set of vertices in $\tQ(\Xi)$ from which three arrows leave.
Set \eqqc{defC}{C_i =\left\{ \arr{ll}{\{ c_i\} \subset k^* & i\in \tQ_{01}(\Xi), \\ \{ c_i, b_i\} \subset k^*& i\in \tQ_{02}(\Xi), \\ \{ c_i, b_i,b'_i\} \subset k^*&i\in \tQ_{03}(\Xi), }\right.}
and set \eqqc{defCp}{C_i' =\left\{ \arr{ll}{\emptyset & i\in \tQ_{01}(\Xi), \\ \{  b_i\} \subset k^*& i\in \tQ_{02}(\Xi), \\ \{ b_i,b'_i\} \subset k^*& i\in \tQ_{03}(\Xi). }\right.}
Let %\tiny
 \eqqc{defU}{U_i(C_i) =\left\{ \arr{ll}{\{ c_i\xc^2_i +\zeta_{j,i}\mu_{i,j}\}  & i\in \tQ_{01}(\Xi), \\ \{ \zeta_{j_1,i}\mu_{i,j_1 } - b_i \zeta_{j_2,i}\mu_{i,j_2 }, \xc^2_i - c_i\zeta_{j_1,i}\mu_{i,j_1}\} & i\in \tQ_{02}(\Xi), \\ \{ b_i \zeta_{j_1,i}\mu_{i,j_1 } - \zeta_{j_2,i}\mu_{i,j_2 }, b'_i \zeta_{j_1,i}\mu_{i,j_1 } - \zeta_{j_3,i}\mu_{i,j_3 }, \\ \quad   \xc^2_i - c_i\zeta_{j_1,i}\mu_{i,j_1} \} & i\in \tQ_{03}(\Xi). }\right.}\normalsize
 and let \small\eqqc{defUm}{U^-_i(C'_i) =\left\{ \arr{ll}{\{ \zeta_{j,i}\mu_{i,j}  \}  & i\in \tQ_{01}(\Xi), \\ \{ \xc^2, \zeta_{j_1,i}\mu_{i,j_1 } - b_i\zeta_{j_2,i}\mu_{i,j_2 }    \} & i\in \tQ_{02}(\Xi), \\ \{ \xc_i^2,  b_i \zeta_{j_1,i}\mu_{i,j_1 } - \zeta_{j_2,i}\mu_{i,j_2 },  b'_i \zeta_{j_1,i}\mu_{i,j_1 } - \zeta_{j_3,i}\mu_{i,j_3 } \} & i\in \tQ_{03}(\Xi). }\right.}\normalsize

\begin{lem}\label{quadraticDualU}
A basis of the orthogonal subspace of $U_i(C_i)$ and of $U^-_i(C'_i)$ are respectively:
\eqqc{defUp}{U_i^{\perp}(C_i) =\left\{ \arr{ll}{\{ \zeta_{j,i}\mu_{i,j} + c_i\xc^2_i \}  & i\in \tQ_{01}(\Xi), \\ \{ \zeta_{j_1,i}\mu_{i,j_1 } + b_i^{-1}\zeta_{j_2,i}\mu_{i,j_2 }+ c^{-1}_i\xc_i^2 \} & i\in \tQ_{02}(\Xi), \\ \{  \zeta_{j_1,i}\mu_{i,j_1 } + b^{-1}_i\zeta_{j_2,i}\mu_{i,j_2 }+ {b'}^{-1}_i  \zeta_{j_3,i}\mu_{i,j_3 } +c_i \xc^2_i \} & i\in \tQ_{03}(\Xi). }\right.}\normalsize
and let \eqqc{defUmp}{U^{-,\perp}_i(C'_i) =\left\{ \arr{ll}{\{ \zeta_{j,i}\mu_{i,j}  \}  & i\in \tQ_{01}(\Xi), \\ \{ b_i \zeta_{j_1,i} \mu_{i,j_1 } +  \zeta_{j_2,i}\mu_{i,j_2 }   \} & i\in \tQ_{02}(\Xi), \\ \{ \zeta_{j_1,i}\mu_{i,j_1 } + b^{-1}_i \zeta_{j_2,i}\mu_{i,j_2 }+ {b'}^{-1}_i  \zeta_{j_3,i}\mu_{i,j_3 }  \} & i\in \tQ_{03}(\Xi). }\right.}
\end{lem}
\begin{proof}
This follow immediately from (2) of Lemma \ref{loop_one_arrow}, (3) of Lemma \ref{loop_two_arrow} and (3) of Lemma \ref{loop_three_arrow}.
\end{proof}

Fix $J\subseteq Q_0(\Xi)$, for a set \eqqc{defCJ}{C_J = C_a \cup \bigcup_{i\in Q_0\setminus J} C_{i}\cup \bigcup_{i\in J} C'_i,} for $C_i$ and $C_i'$ as defined in \eqref{defC} and \eqref{defCp}, set
\eqqc{deftrho}{\trho(\Xi,J,C_J) = \trho_p(\Xi) \cup \trho_a(\Xi,C_a) \cup \bigcup_{i\in \tQ_0\setminus J} U_i (C_i)\cup \bigcup_{i\in J} U^-_i (C'_i).}

By Lemma \ref{orthogonal_in_ij} and Lemma \ref{quadraticDualU}, we have the following.
\begin{pro}\label{qdrels}
\eqqc{deftrhop}{\trho^{\perp}(\Xi,J,C_J) ={\trho_a}^\perp(\Xi,C_a) \cup \bigcup_{i\in \tQ_0\setminus J} U^{\perp}_i (C_i) \cup \bigcup_{i\in J} U^{-,\perp}_i (C_i).}
\end{pro}

\begin{pro}\label{relQed}
Let $\Xi$ be $A_l$, $D_l$, with $l\ge 5$, and $E_6, E_7, E_8$.
If $\tL'(\Xi)=\tQ(\Xi)/I(\Xi)$ is  $2$-translation algebra then there is a subset $J$ of $\tQ_0$ and a set $C_J$ as in \eqref{defC} of nonzero elements in $k$, such that by choosing the representatives of the arrows properly, we have $I(\Xi)$ is generated by $\trho(\Xi,J,C_J)$.
\end{pro}
\begin{proof}
Take  $J=\{i\in \tQ_0(\Xi)| \xc^2_i\in I(\Xi)\}$.
Then by choosing the representatives of the arrows properly, we have that there is a set $C_J\subset k^*$ as in \eqref{defC}, such that $\trho(\Xi,J,C_J)\subset I(\Xi)$, by Propositions \ref{rho_diffih}, \ref{rho_onearrow} and Lemmas \ref{loop_one_arrow}, \ref{loop_two_arrow} and \ref{loop_three_arrow}.

Let $\tL = k\tQ(\Xi)/(\trho(\Xi,J,C_J))$.
By Lemma \ref{diffih}, we have $\dmn_k e_j \tL_1 e_i = \dmn_k e_i \tL_2 e_j =0$ if there is no arrow from $i$ to $j$, by Lemma \ref{orthogonal_in_ij}, we have $\dmn_k e_j \tL_1 e_i = \dmn_k e_i \tL_2 e_j =1$ if there is an arrow from $i$ to $j$ and by (2) of Lemma \ref{loop_one_arrow}, (3) of Lemma \ref{loop_two_arrow} and (3) of Lemma \ref{loop_three_arrow}, and $\dmn_k e_i \tL_1 e_i = \dmn_k e_i \tL_2 e_i =1$ for all $i$.
Since $k\tQ(\Xi)/I(\Xi)$ is $2$-translation and the Loewy matrix of $k\tQ(\Xi)/I(\Xi)$ is \eqref{LoewyL} with $\tQ = \tQ(\Xi)$ and $\trho(\Xi,J,C_J)\subset I(\Xi)$, this implies that $\trho(\Xi,J,C_J)$ is a generating set of $I(\Xi)$.
\end{proof}

It is also immediate that $\trho^{\perp}(\Xi,J,C_J)$ spans the orthogonal spaces of $\trho(\Xi,J,C_J)$ in $k\tQ_2$.
So we have the following characterization of relations for $\tQ(\Xi)$.

By Lemma \ref{qdrels}, a quadratic dual relation of $\trho(\Xi,J,C)$ is $\trho^{\perp}(\Xi,J,C)$ defined in \eqref{deftrhop}.
So we get the following, similar to the proof of Proposition \ref{relQsrAS}.
\begin{pro}\label{relQedp}
Let $\Xi$ be $A_l$, $D_l$, with $l\ge 4$, and $E_6, E_7, E_8$.
If $\tG = \tQ(\Xi)/I$ is Koszul AS-regular algebra of finite Gelfand-Kirilov dimension then there is a subset $J$ of $\tQ_0$ and a set $C$ as in \eqref{defC} of nonzero elements in $k$, such that  $\tL(\Xi) \simeq k \tQ(\Xi)/(\trho^{\perp}(\Xi,J,C))$.
\end{pro}

\medskip

Now construct the nicely-graded quiver $\zZ_{\vv}\tQ(\Xi)$ over $\tQ(\Xi)$,  it has only one connected component since $\tQ(\Xi)$ has loops.
With the relation set \eqqcn{}{\rho_{\zZ_{\vv}\tQ(\Xi)} = \rho_{\zZ_{\vv}\tQ(\Xi)}(J,C_J)=\{z[m]| z\in \trho(\Xi,J,C_J),t\in \zZ\}} for some parameter set $C_J = C(\Xi,J)$, $\zZ_{\vv}\tQ(\Xi)$ forms a $2$-translation quiver with $2$-translation algebra $$\LL(\Xi,J,C) \simeq k\zZ_{\vv}\tQ(\Xi)/(\rho_{\zZ_{\vv}\tQ(\Xi)}(J,C_J)) ,$$ if $\tL(\Xi)$ is an $2$-translation algebra.

By taking the complete $\tau$-slices $Q(\Xi)=\zZ_{\vv}\tQ(\Xi)[0,2]$, we obtain the following quivers:
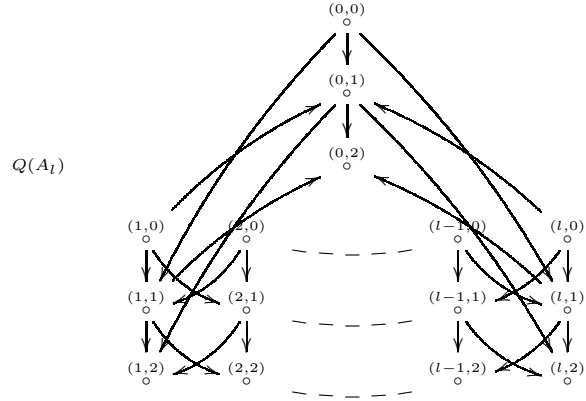
\tiny$$\xymatrix@C=0.6cm@R0.4cm{
&&&\stackrel{(0,0)}{\circ}\ar@[purple][d]\ar@/_/@[green][ddddll]\ar@/^/[ddddrr]&&&\\
&&&\stackrel{(0,1)}{\circ}\ar@[purple][d]\ar@/_/@[green][ddddll]\ar@/^/[ddddrr]&&&\\
Q(A_l) &&&\stackrel{(0,2)}{\circ} &&&\\
& \stackrel{(1,0)}{\circ}\ar@[purple][d] \ar@/_/@[green][dr]\ar@/^/[uurr] & \stackrel{(2,0)}{\circ}\ar@[purple][d]\ar@/^/[dl]\ar@/_/@{--}[rr] & &\stackrel{(l-1,0)}{\circ}\ar@[purple][d] \ar@/_/@[green][dr] &\stackrel{(l,0)}{\circ}\ar@[purple][d] \ar@/^/[dl]\ar@/_/@[green][uull]\\
& \stackrel{(1,1)}{\circ}\ar@[purple][d] \ar@/_/@[green][dr]\ar@/^/[uurr] & \stackrel{(2,1)}{\circ}\ar@[purple][d]\ar@/^/[dl]\ar@/_/@{--}[rr] & &\stackrel{(l-1,1)}{\circ}\ar@[purple][d] \ar@/_/@[green][dr] &\stackrel{(l,1)}{\circ}\ar@[purple][d] \ar@/^/[dl]\ar@/_/@[green][uull]\\
& \stackrel{(1,2)}{\circ}  & \stackrel{(2,2)}{\circ}\ar@/_/@{--}[rr] & &\stackrel{(l-1,2)}{\circ}  &\stackrel{(l,2)}{\circ} \\
&&&&&&&\\
}$$\normalsize
\tiny
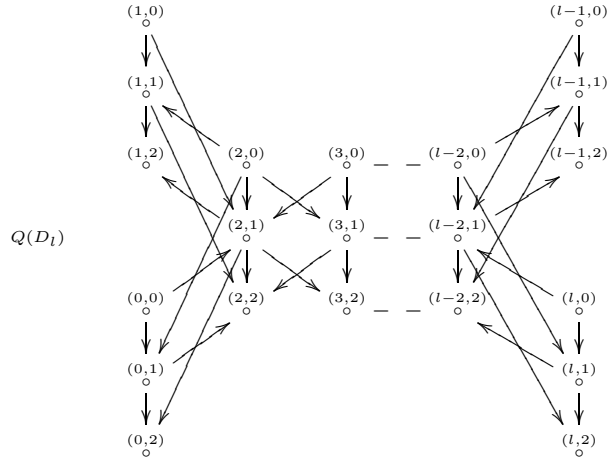
$$\xymatrix@C=0.6cm@R0.4cm{
&\stackrel{(1,0)}{\circ}\ar@[purple][d]\ar@[green][dddr]&&&&\stackrel{(l-1,0)}{\circ}\ar@[purple][d] \ar[dddl]&&\\
&\stackrel{(1,1)}{\circ}\ar@[purple][d]\ar@[green][dddr]&&&&\stackrel{(l-1,1)}{\circ}\ar@[purple][d] \ar[dddl]&&\\
&\stackrel{(1,2)}{\circ} &\stackrel{(2,0)}{\circ}\ar@[purple][d]\ar[ul]\ar[dddl] \ar@[green][dr] &\stackrel{(3,0)}{\circ}\ar@[purple][d]\ar[dl] \ar@{--}[r] &\stackrel{(l-2,0)}{\circ}\ar@[purple][d]\ar@[green][ur] \ar@[green][dddr] &\stackrel{(l-1,2)}{\circ} \\
Q(D_l) & &\stackrel{(2,1)}{\circ}\ar@[purple][d]\ar[ul]\ar[dddl] \ar@[green][dr] &\stackrel{(3,1)}{\circ}\ar@[purple][d]\ar[dl] \ar@{--}[r] &\stackrel{(l-2,1)}{\circ}\ar@[purple][d]\ar@[green][ur] \ar@[green][dddr] &\\
&\stackrel{(0,0)}{\circ}\ar@[purple][d]\ar@[green][ur] &\stackrel{(2,2)}{\circ} & \stackrel{(3,2)}{\circ} \ar@{--}[r] &\stackrel{(l-2,2)}{\circ} &\stackrel{(l,0)}{\circ}\ar@[purple][d]\ar[ul] \\
&  \stackrel{(0,1)}{\circ}\ar@[purple][d]\ar@[green][ur]& & & &\stackrel{(l,1)}{\circ}\ar@[purple][d] \ar[ul]&&\\
&  \stackrel{(0,2)}{\circ} & & & &\stackrel{(l,2)}{\circ}&&\\
&&&&&&&\\
}$$\normalsize
\tiny
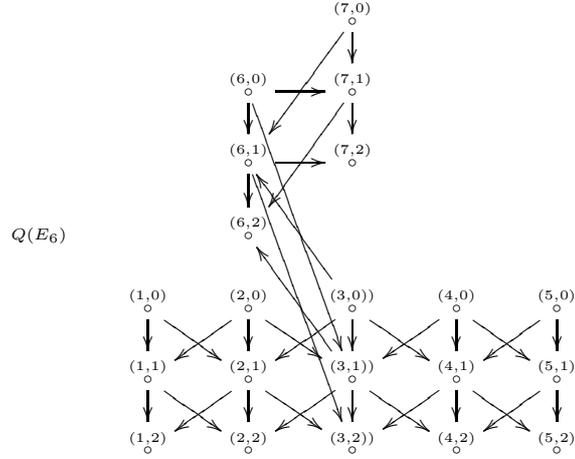
$$\xymatrix@C=0.6cm@R0.4cm{
&&&\stackrel{(7,0)}{\circ}\ar@[purple][d]\ar[ddl]&&&\\ &&\stackrel{(6,0)}{\circ}\ar@[purple][d]\ar@[green][r]\ar[ddddr]& \stackrel{(7,1)}{\circ}\ar@[purple][d]\ar[ddl]&&&\\
&&\stackrel{(6,1)}{\circ}\ar@[purple][d]\ar[ddddr]\ar@[green][r]& \stackrel{(7,2)}{\circ}&&&\\
Q(E_6) &&\stackrel{(6,2)}{\circ}&&&&\\
& \stackrel{(1,0)}{\circ}\ar@[purple][d] \ar@[green][dr] & \stackrel{(2,0)}{\circ}\ar@[purple][d]\ar[dl]\ar@[green][dr] & \stackrel{(3,0))}{\circ}\ar@[purple][d]\ar@[green][uul] \ar@[green][dr] \ar[dl] &\stackrel{(4,0)}{\circ}\ar@[purple][d] \ar@[green][dr] \ar[dl] &\stackrel{(5,0)}{\circ}\ar@[purple][d] \ar[dl]\\
& \stackrel{(1,1)}{\circ}\ar@[purple][d] \ar@[green][dr] & \stackrel{(2,1)}{\circ}\ar@[purple][d]\ar[dl]\ar@[green][dr] & \stackrel{(3,1))}{\circ}\ar@[purple][d]\ar@[green][uul] \ar@[green][dr] \ar[dl] &\stackrel{(4,1)}{\circ}\ar@[purple][d] \ar@[green][dr] \ar[dl] &\stackrel{(5,1)}{\circ}\ar@[purple][d] \ar[dl]\\
& \stackrel{(1,2)}{\circ}& \stackrel{(2,2)}{\circ} & \stackrel{(3,2))}{\circ} &\stackrel{(4,2)}{\circ} &\stackrel{(5,2)}{\circ}\\
&&&&&&&\\}$$ \normalsize
\tiny
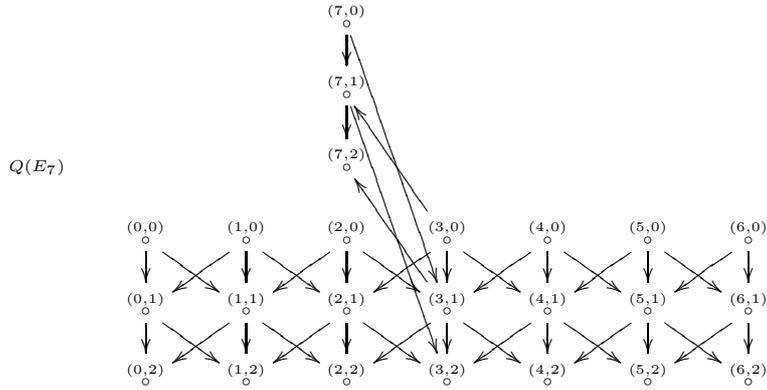
$$ \xymatrix@C=0.6cm@R0.4cm{
&&&\stackrel{(7,0)}{\circ}\ar@[purple][d]\ar[ddddr]&&&&\\
&&&\stackrel{(7,1)}{\circ}\ar@[purple][d]\ar[ddddr]&&&&\\
Q(E_7) &&&\stackrel{(7,2)}{\circ}&&&&\\
& \stackrel{(0,0)}{\circ}\ar@[purple][d] \ar@[green][dr] & \stackrel{(1,0)}{\circ}\ar@[purple][d]\ar[dl]\ar@[green][dr] & \stackrel{(2,0)}{\circ}\ar@[purple][d] \ar@[green][dr] \ar[dl] &\stackrel{(3,0)}{\circ}\ar@[purple][d]\ar@[green][uul] \ar@[green][dr] \ar[dl]&\stackrel{(4,0)}{\circ}\ar@[purple][d] \ar@[green][dr] \ar[dl] &\stackrel{(5,0)}{\circ}\ar@[purple][d] \ar@[green][dr] \ar[dl]&\stackrel{(6,0)}{\circ}\ar@[purple][d] \ar[dl]\\
& \stackrel{(0,1)}{\circ}\ar@[purple][d] \ar@[green][dr] & \stackrel{(1,1)}{\circ}\ar@[purple][d]\ar[dl]\ar@[green][dr] & \stackrel{(2,1)}{\circ}\ar@[purple][d] \ar@[green][dr] \ar[dl] &\stackrel{(3,1)}{\circ}\ar@[purple][d]\ar@[green][uul] \ar@[green][dr] \ar[dl]&\stackrel{(4,1)}{\circ}\ar@[purple][d] \ar@[green][dr] \ar[dl] &\stackrel{(5,1)}{\circ}\ar@[purple][d] \ar@[green][dr] \ar[dl]&\stackrel{(6,1)}{\circ}\ar@[purple][d] \ar[dl]\\
& \stackrel{(0,2)}{\circ}  & \stackrel{(1,2)}{\circ}  & \stackrel{(2,2)}{\circ}  &\stackrel{(3,2)}{\circ} &\stackrel{(4,2)}{\circ}  &\stackrel{(5,2)}{\circ} &\stackrel{(6,2)}{\circ} \\
&&&&&&&\\
}$$\normalsize
 \tiny
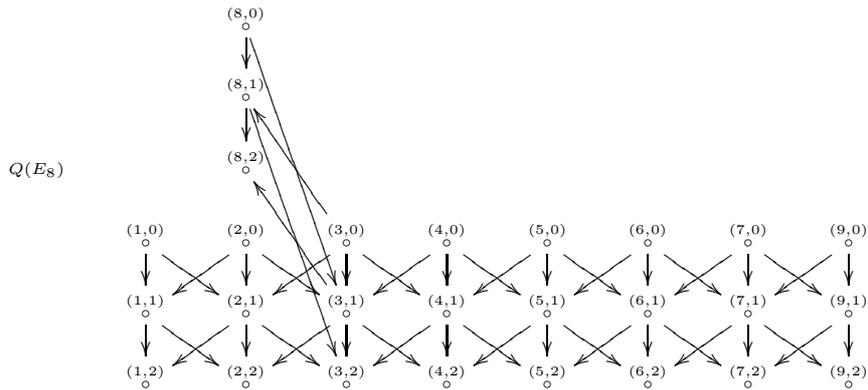
$$ \xymatrix@C=0.6cm@R0.4cm{
&&\stackrel{(8,0)}{\circ}\ar@[purple][d]\ar[ddddr]&&&&\\
&&\stackrel{(8,1)}{\circ}\ar@[purple][d]\ar[ddddr]&&&&\\
Q(E_8) &&\stackrel{(8,2)}{\circ}&&&&\\
& \stackrel{(1,0)}{\circ}\ar@[purple][d] \ar@[green][dr] & \stackrel{(2,0)}{\circ}\ar@[purple][d]\ar[dl]\ar@[green][dr] & \stackrel{(3,0)}{\circ}\ar@[purple][d]\ar@[green][uul] \ar@[green][dr] \ar[dl] &\stackrel{(4,0)}{\circ}\ar@[purple][d] \ar@[green][dr] \ar[dl] &\stackrel{(5,0)}{\circ}\ar@[purple][d] \ar@[green][dr] \ar[dl]&\stackrel{(6,0)}{\circ}\ar@[purple][d] \ar@[green][dr] \ar[dl] &\stackrel{(7,0)}{\circ}\ar@[purple][d] \ar@[green][dr] \ar[dl]&\stackrel{(9,0)}{\circ}\ar@[purple][d] \ar[dl]\\
& \stackrel{(1,1)}{\circ}\ar@[purple][d] \ar@[green][dr] & \stackrel{(2,1)}{\circ}\ar@[purple][d]\ar[dl]\ar@[green][dr] & \stackrel{(3,1)}{\circ}\ar@[purple][d]\ar@[green][uul] \ar@[green][dr] \ar[dl] &\stackrel{(4,1)}{\circ}\ar@[purple][d] \ar@[green][dr] \ar[dl] &\stackrel{(5,1)}{\circ}\ar@[purple][d] \ar@[green][dr] \ar[dl]&\stackrel{(6,1)}{\circ}\ar@[purple][d] \ar@[green][dr] \ar[dl] &\stackrel{(7,1)}{\circ}\ar@[purple][d] \ar@[green][dr] \ar[dl]&\stackrel{(9,1)}{\circ}\ar@[purple][d] \ar[dl]\\
& \stackrel{(1,2)}{\circ}  & \stackrel{(2,2)}{\circ}  & \stackrel{(3,2)}{\circ}  &\stackrel{(4,2)}{\circ}  &\stackrel{(5,2)}{\circ}   &\stackrel{(6,2)}{\circ} &\stackrel{(7,2)}{\circ}  &\stackrel{(9,2)}{\circ}
}
$$\normalsize

They are all nicely-graded quivers.
We get $n$-properly-graded quivers by taking the relation \eqqc{relnpg}{\rho(\Xi,J,C) = \{z[0]| z\in \trho(\Xi, J, C)\}.}
For any parameter set $C_J$, we can replace the representatives for the arrow $\xa_{i,j,1},\xb_{j,i,1}$ and $\xc_{i,1}$ such that all the parameters become $1$.

Write \eqqcn{defmuzetat}{\mu_{i,j,t}= \left\{\arr{ll}{\xa_{i,j,t} & i<j, t=0,1,\\ \xb_{i,j,t} & i>j, t=0,1, }\right. \mbox{ and } \zeta_{j,i}= \left\{\arr{ll}{\xb_{j,i,t} & i<j,t=0,1,\\ \xa_{j,i,t} & i>j, t=0,1. }\right.}

Let %\tiny
 \eqqc{defUS}{U_i(C_i) =\left\{ \arr{ll}{\{ \xc_{i,1} \xc_{i,0} +\zeta_{j,i,1}\mu_{i,j,0}\}  & i\in \tQ_{01}(\Xi), \\ \{ \zeta_{j_1,i,1}\mu_{i,j_1,0 } - \zeta_{j_2,i,1}\mu_{i,j_2,0 }, \xc_{i,1} \xc_{i,0}  - \zeta_{j_1,i,1}\mu_{i,j_1,0}\} & i\in \tQ_{02}(\Xi), \\ \{ \zeta_{j_1,i,1}\mu_{i,j_1,0 } - \zeta_{j_2,i,1}\mu_{i,j_2,0 },  \zeta_{j_1,i,1}\mu_{i,j_1,0 } - \zeta_{j_3,i,1}\mu_{i,j_3,0 }, \\ \quad   \xc_{i,1} \xc_{i,0}  - \zeta_{j_1,i,1}\mu_{i,j_1,0} \} & i\in \tQ_{03}(\Xi). }\right.}\normalsize
 and let \eqqc{defUm}{U^-_i(C'_i) =\left\{ \arr{ll}{\{ \zeta_{j,i,1}\mu_{i,j,0}  \}  & i\in \tQ_{01}(\Xi), \\ \{ \xc^2, \zeta_{j_1,i,1}\mu_{i,j_1,0 } - \zeta_{j_2,i,1}\mu_{i,j_2,0 }    \} & i\in \tQ_{02}(\Xi), \\ \{   \zeta_{j_1,i,1}\mu_{i,j_1,0 } +  \zeta_{j_2,i,1}\mu_{i,j_2,0 }+   \zeta_{j_3,i,1}\mu_{i,j_3,0 }  \} & i\in \tQ_{03}(\Xi). }\right.}

\eqqc{defUpS}{U_i^{\perp}(C_i) =\left\{ \arr{ll}{\{ \zeta_{j,i,1}\mu_{i,j,0} + \xc_{i,1} \xc_{i,0} \}  & i\in \tQ_{01}(\Xi), \\ \{ \zeta_{j_1,i,1}\mu_{i,j_1,0 } + \zeta_{j_2,i,1}\mu_{i,j_2,0 }+  \xc_{i,1} \xc_{i,0} \} & i\in \tQ_{02}(\Xi), \\ \{  \zeta_{j_1,i,1}\mu_{i,j_1,0 } + b^{-1}_i\zeta_{j_2,i,1}\mu_{i,j_2,0 }+ \zeta_{j_3,i,1}\mu_{i,j_3,0 } + \xc_{i,1} \xc_{i,0} \} & i\in \tQ_{03}(\Xi). }\right.}\normalsize
and let \eqqc{defUmp}{U^{-,\perp}_i(C'_i) =\left\{ \arr{ll}{\{ \zeta_{j,i,1}\mu_{i,j,0}  \}  & i\in \tQ_{01}(\Xi), \\ \{  \zeta_{j_1,i,1} \mu_{i,j_1,0 } +  \zeta_{j_2,i,1}\mu_{i,j_2,0 }   \} & i\in \tQ_{02}(\Xi), \\ \{ \zeta_{j_1,i,1}\mu_{i,j_1,0 } + \zeta_{j_2,i,1}\mu_{i,j_2,0 }+   \zeta_{j_3,i,1}\mu_{i,j_3,0 }  \} & i\in \tQ_{03}(\Xi). }\right.}

Take \eqqcn{twoarrowsS}{ \arr{rl}{\rho_p(\Xi)
= & \{\xa_{i,j,1}\xa_{j,h,0} || i,h ,j \in \tQ_0(\Xi), i\neq h \} \cup \{ \xb_{i,j,1}\xb_{j,h,0} | i,h ,j \in \tQ_0(\Xi), i\neq h  \} \\ &\cup \{ \xb_{i,j,1}\xa_{j,h,0}| i,h ,j \in \tQ_0(\Xi), i\neq h  \}\cup  \{ \xb_{i,j,1}\xa_{j,h,0} | i,h ,j \in \tQ_0(\Xi), i\neq h\} ,}}
and \eqqcn{onearrowS}{\rho_a(\Xi) =\{\xa_{i,j,1}\xc_{i,0} - \xc_{j,1}\xa_{i,j,0},  \xb_{j,i,1} \xc_{j,0} - \xc_{i,1}\xb_{j,i,0} | \xa_{i,j}\in \tQ_1, i<j \}.}
Write \eqqcn{onearrowpS}{{\rho_a}^\perp(\Xi) =\{\xa_{i,j,1}\xc_{i,0} + \xc_{j,1}\xa_{i,j,0},  \xb_{j,i,1}\xc_{j,0} + \xc_{i,1}\xb_{j,i,0} | \xa_{i,j}\in \tQ_1, i<j\}.}
They are subsets of the space $k\tQ_2(\Xi)$ spanned by the paths of length $2$, and $\rho_a(\Xi)$ and ${\rho_a}^\perp(\Xi)$ span orthogonal subspaces in the subspace spanned by $\rho_a(\Xi)\cup {\rho_a}^\perp(\Xi)$.

Take a subset $J\subseteq \tQ_0(\Xi)$.
Let \eqqc{defrho}{\rho(\Xi,J) = \rho_p(\Xi) \cup \rho_a(\Xi) \cup \bigcup_{i\in \tQ_0\setminus J} U_{i,0} \cup \bigcup_{i\in J} U^-_{i,0},}  and  \eqqc{deftrhop2sa}{\rho^{\perp}(\Xi,J) ={\rho_a}^\perp(\Xi) \cup \bigcup_{i\in \tQ_0\setminus J} U^{\perp}_{i,0} \cup \bigcup_{i\in J} U^{-,\perp}_{i,0}.}

We have the following descriptions of the relation sets for the $\tau$-slice algebras and for the $n$-slice algebras associated to the McKay quiver $\tQ(\Xi)$.
\begin{pro}\label{relQtauslc}
Let $\Xi$ be $A_l$, $D_l$, with $l\ge 4$, and $E_6, E_7, E_8$.
If $\LL=kQ(\Xi)/I$ is  $\tau$-slice algebra with $\dtL$ a $2$-translation algebra then there is a subset $J$ of $\tQ_0(\Xi)$  such that  $I=(\rho(\Xi,J))$.
\end{pro}

\begin{pro}\label{nslcalg}
Let $\Xi$ be $A_l$, $D_l$, with $l\ge 4$, and $E_6, E_7, E_8$.
If $\GG =kQ(\Xi)/I$ is  $2$-slice algebra, then  there is a subset $J$ of $\tQ_0(\Xi)$  such that  $I=(\rho^{\perp}(\Xi,J))$.
\end{pro}

\subsection{Remarks.}
\begin{enumerate}
\item By constructing $\zZ_{\vv}\tQ$ for given $\tQ(G), J, C$, and taking a complete $\tau$-slice $Q$ in one of its connected component, we get an indecomposable $\tau$-slice algebra $\LL = kQ/(\rho)$, and a $2$-slice algebra $kQ/(\rho^{\perp})$.
    For a given complete $\tau$-slice, we can change the representatives of the arrows so the relation set have the same expressions, so the $\tau$-slice algebras and the $2$-slice algebras are isomorphic, respectively (for a fixed subset $J$ of the vertices in the case $\tQ=\tQ(\Xi)$).
    The algebra $\LL$ dependents only on the quiver $Q$ and the set $J$, but is independent of the parameters in $C$.

\item    The complete $\tau$-slices in  $\zZ_{\vv}\tQ$ are usually non-isomorphic as quivers.
    But they all can be obtained from any one of them by a sequence of $\tau$-mutations by \cite{g12}.
    By \cite{gx20}, the corresponding $2$-slice algebras are obtained by a sequence of $2$-APR tilts from the one described in Propositions \ref{2slcrelp} and  \ref{nslcalg}.

\item Though we need the field $k$ containing $\mathbb C$ for defining the McKay quiver, our results in this section are valid for any field, we only need the McKay quiver and the combinatory data in the Loewy matrix related to the McKay quiver, at least in case of $n=2$.
\end{enumerate}

{\bf Acknowledgements.}
%\begin{acknowledgements}
We would like to thank  the referees  for reading the manuscript carefully and for suggestions and comments on revising and improving the paper.
They also thank the  referee for bring \cite{llmss20} to their attention.
%\end{acknowledgements}
{}

\end{document}